\documentclass{amsart}

\usepackage{amssymb}
\usepackage{amsmath}
\usepackage{amsthm}
\usepackage{amsbsy}
\usepackage{bm}
\usepackage{hyperref}
\date{\today}
\usepackage{cite}
\usepackage{array}
\usepackage{color}
\usepackage{tikz}
\usepackage{graphics}
\usepackage{subcaption}

\theoremstyle{theorem}
    \newtheorem{theorem}{Theorem}
    \newtheorem{lemma}[theorem]{Lemma}
    \newtheorem{proposition}[theorem]{Proposition}
    \newtheorem{corollary}[theorem]{Corollary}

\theoremstyle{definition} 
    \newtheorem{definition}[theorem]{Definition}
    
    \newtheorem{result}[theorem]{Result}
    \newtheorem{remark}[theorem]{Remark}
    \newtheorem{example}[theorem]{Example}
    \newtheorem{exercise}[theorem]{Exercise}

\def\suchthat{\; : \;}

\def\sig{\sigma}

\def\Z{\mathbb{Z}}

\def\Z{\mathbb{Z}}

\def\tends{\rightarrow}

\def\<{\langle}
\def\>{\rangle}

\newcommand{\E}{\mbox{E}}

\def\P{\mbox{P}}

\newcommand\Tr{{\mbox{Tr}}}

\newcommand\mnote[1]{} 
\newcommand\be{\begin{equation*}}

\newcommand\ee{\end{equation*}}

\newcommand\ben{\begin{equation}}
\newcommand\een{\end{equation}}
\newcommand\bes{\begin{eqnarray*}}
\newcommand\ees{\end{eqnarray*}}

\newcommand\bex{\begin{exercise}}
\newcommand\eex{\end{exercise}}
\newcommand\beg{\begin{example}}
\newcommand\eeg{\end{example}}
\newcommand\benu{\begin{enumerate}}
\newcommand\eenu{\end{enumerate}}
\newcommand\beit{\begin{itemize}}
\newcommand\eeit{\end{itemize}}
\newcommand\berk{\begin{remark}}
\newcommand\eerk{\end{remark}}
\newcommand\bdefn{\begin{defintion}}
\newcommand\edefn{\end{definition}}
\newcommand\bthm{\begin{theorem}}
\newcommand\ethm{\end{theorem}}
\newcommand\bprf{\begin{proof}}
\newcommand\eprf{\end{proof}}
\newcommand\blem{\begin{lemma}}
\newcommand\elem{\end{lemma}}

\newcommand{\Cov}{\mbox{\rm Cov}}
\newcommand{\sm}{{\raise0.3ex\hbox{$\scriptstyle \setminus$}}}

\def\sig{\sigma}

\def\tends{\rightarrow}

\def\CHI{\mathchoice%
{\raise2pt\hbox{$\chi$}}%
{\raise2pt\hbox{$\chi$}}%
{\raise1.3pt\hbox{$\scriptstyle\chi$}}%
{\raise0.8pt\hbox{$\scriptscriptstyle\chi$}}}
\def\smalloplus{\raise1pt\hbox{$\,\scriptstyle \oplus\;$}}

\title[Linear eigenvalue statistics of random circulant matrices]{Process convergence of Fluctuations of linear eigenvalue statistics of random circulant matrices }

\author{Shambhu Nath Maurya}
\address{Department of Mathematics\\
        Indian Institute of Technology Bombay\\
         Powai, Mumbai, Maharashtra 400076, India}

\email{snmaurya [at] math.iitb.ac.in}

\author{Koushik Saha}
\address{Department of Mathematics\\
        Indian Institute of Technology Bombay\\
         Powai, Mumbai, Maharashtra 400076, India}

\email{koushik.saha [at] iitb.ac.in}

\date{\today}
\thanks{This work is partially supported by UGC Doctoral Fellowship, India of Shambhu Nath Maurya.}
\begin{document}

\begin{abstract}
In this paper we discuss the process convergence of the time dependent fluctuations of linear eigenvalue  statistics of random circulant matrices with independent Brownian motion entries, as the dimension of the matrix
tends to $\infty $. Our derivation is based on the trace formula of circulant matrix,  method of moments and some combinatorial techniques.
\end{abstract}

\maketitle

\section{ introduction and main results}
Suppose $M_n$ is an $n\times n$ matrix with real or complex entries. {\it linear eigenvalue  statistic} of
of $M_n$ is a
function of the form  
\begin{equation} \label{eq: 1}
\mathcal M_n(f)=\sum_{k=1}^{n}f(\lambda_k)
\end{equation}
where $\lambda_1,\lambda_2,\ldots, \lambda_n$ are the eigenvalues  of $A_n$ and $f$ is some  fixed function. The function $f$ is known as the test function.
 
The study of the linear eigenvalue statistics of random matrices is one of popular area of research in random matrix theory. The literature on linear eigenvalue statistics of random matrices is quite big. To the best of our knowledge, the fluctuation problem of linear eigenvalue statistics of random matrices was first considered by  Arharov \cite{arharov} in 1971.   
In 1982, Jonsson \cite{jonsson} established the Central limit theorem
(CLT) type results  of linear eigenvalue statistics for Wishart matrices using method of moments.
   
In last three decades the fluctuations of linear eigenvalue statistics have been  extensively studied  for various random matrix models.     
To get an overview on the fluctuation results of  Wigner and sample covariance matrices,   we refer the readers to   \cite{johansson1998}, \cite{soshnikov1998tracecentral}, \cite{bai2004clt}, \cite{lytova2009central}, \cite{shcherbina2011central}, \cite{sosoe_wong_2013} and the reference there in.  
For results on band and sparse Wigner  matrices, see  \cite{anderson2006clt},  \cite{jana2014}, \cite{li2013central} and  \cite{shcherbina2015}, and for non-Hermitian matrices, see \cite{rider_silverstein_2006}, \cite{nourdin_peccati_2010} and \cite{rourke_renfrew_2016}. 

For  fluctuations of linear eigenvalue statistcs  of other patterned random matrices, namely, Toeplitz, Hankel, band Toeplitz matrices and circulant type matrices, see  \cite{chatterjee2009fluctuations}, \cite{liu2012fluctuations}, \cite{adhikari_saha2017} and \cite{adhikari_saha2018}.

The fluctuation problem, we are interested to consider in this article, is inspired by the  results on band Toeplitz matrix by Li and Sun \cite{li_sun_2015}.
They  studied time dependent fluctuations of linear eigenvalue statistics for band Toeplitz matrices with  standard Brownian motion entries.
Here we study similar type of time dependent fluctuations of linear eigenvalue statistics for random circulant matrices. 
 
A circulant matrix is defined as 
$$
C_n=\left(\begin{array}{cccccc}
x_0 & x_1 & x_2 & \cdots & x_{n-2} & x_{n-1} \\
x_{n-1} & x_0 & x_1 & \cdots & x_{n-3} & x_{n-2}\\
\vdots & \vdots & {\vdots} & \ddots & {\vdots} & \vdots \\
x_1 & x_2 & x_3 & \cdots & x_{n-1} & x_0
\end{array}\right).
$$
Observe that for $j=1,2,\ldots, n-1$,  $(j+1)$-th row of circulant matrix is obtained by giving its $j$-th row a right circular shift by one positions and the (i,\;j)-th element of the matrix is $x_{(j-i) \mbox{ \tiny{mod} } n}$. 

Here we consider  time dependent random circulant matrices whose entries are coming from a sequence of independent standard Brownian motions (SBM) $\{b_n(t);t\geq 0\}_{n\geq 0}$. At time $t$, the entries $\{x_0,x_1,\ldots,x_{n-1}\}$ of $C_n$ will be $\{\frac{b_0(t)}{\sqrt n},\frac{b_1(t)}{\sqrt n},\ldots,\frac{b_{n-1}(t)}{\sqrt n}\}$.  We  denote the  circulant matrix at time $t$ by $C_n(t)$. Therefore 
\begin{equation}\label{def:C_n(t)}
C_n(t)=\frac{1}{\sqrt n}\left(\begin{array}{cccccc}
b_0(t) & b_1(t) & b_2(t) & \cdots & b_{n-2}(t) & b_{n-1}(t) \\
b_{n-1}(t) & b_0(t) & b_1(t) & \cdots & b_{n-3}(t) & b_{n-2}(t)\\
\vdots & \vdots & {\vdots} & \ddots & {\vdots} & \vdots \\
b_1(t) & b_2(t) & b_3(t) & \cdots & b_{n-1}(t) & b_0(t)
\end{array}\right).
\end{equation}
Now we consider linear eigenvalue statistics as defined in \eqref{eq: 1} for $C_n(t)$ with test function $f(x)=x^p$, $p\geq 2$. Therefore
$$ \sum_{k=1}^{n}f(\lambda_k(t))= \sum_{k=1}^{n}(\lambda_k(t))^p= \Tr(C_n(t))^p,$$
where $\lambda_1(t),\lambda_2(t),\ldots,\lambda_k(t)$ are the eigenvalues of $C_n(t)$. We scale and center $\Tr(C_n(t))^p$ to study its fluctuation, and define 
\begin{equation}\label{eqn:w_p(t)}
   w_p(t) := \frac{1}{\sqrt{n}} \bigl\{ \Tr(C_n(t))^p - \E[\Tr(C_n(t))^p]\bigr\}. 
   \end{equation}
   Note that $w_p(t)$ depends on $n$. But we suppress $n$  to keep the notation simple. Observe that $\{ w_p(t); t \geq 0\}$ is a continuous stochastic process.
Now we state our main results. The following theorem describes the covariance structure of $w_p(t_1)$ and $w_q(t_2)$ as $n\to\infty$.
\begin{theorem}\label{thm:circovar} 
For $0<t_1\leq t_2$ and $p,q\geq 2$, 
\begin{equation}\label{eqn:condition}
 \lim_{n\to \infty} \Cov \big( w_{p}(t_1),w_{q}(t_2) \big)
   = \left\{\begin{array}{ccc} 	 
		 	 (t_1)^pp!\ \sum_{s=0}^{p-1}	f_p(s) & \text{if}&p=q\\\\
			0 & \text{if}&p\neq q, 	 	 
		 	  \end{array}\right.	
 \end{equation}
where
\begin{equation}\label{eqn:f_p(s)}
 f_{p}(s)=\sum_{k=0}^{s}(-1)^k\binom{p}{k}(s-k)^{p-1}.
\end{equation}
\end{theorem}
The following theorem describes joint convergence of $\{w_p(t);p\geq 2\}$ as $n \to \infty$.

\begin{theorem}\label{thm:cirmulti} 
	As $n\to\infty$, $\{w_p(t);t\geq 0, p\geq 2\}$ jointly converge to an independent family of Gaussian processes $\{N_p(t);t\geq 0,  p\geq 2\}$ in the following sense. Suppose  $\{p_1, p_2, \ldots, p_r\}\subset \mathbb N$ and $ p_i\geq  2$ for $1\leq i\leq r$;  $0<t_1<t_2 \cdots <t_r$ and $\{a_1, a_2, \ldots, a_r \} \subset \mathbb{R}$. 
	Then
	\begin{align}\label{eqn:condition}
	\lim_{n\to \infty} & \P  \Big( w_{p_1}(t_1)  \leq a_1,       w_{p_2}(t_2)  \leq a_2,\ldots, w_{p_r}(t_r)  \leq a_r \Big) \nonumber \\
	&= \P  \Big( N_{p_1}(t_1)  \leq a_1,                              N_{p_2}(t_2)  \leq a_2, \ldots, N_{p_r}(t_r)  \leq a_r \Big), 
	\end{align}
	where $\{N_p(t);t\geq 0, p\geq 2\}$ has mean zero and following covariance structure:
	\begin{equation}\label{eqn:N_p(t):cov}
	\E \Big[ N_p(t_1)N_{q}(t_2) \Big]
	= \left\{\begin{array}{ccc} 	 
	(t_1)^pp!\ \sum_{s=0}^{p-1}	f_p(s) & \text{if}&p=q\\\\
	0 & \text{if}&p\neq q, 	 	 
	\end{array}\right.	
	\end{equation}
	and $f_{p}(s)$  is as in \eqref{eqn:f_p(s)}.
\end{theorem}
In \cite{li_sun_2015}, Li and Sun have studied the joint fluctuation of $(  w_{p_1}(t_1), w_{p_2}(t_2),\ldots, w_{p_r}(t_r))$, where $w_{p_i}(t_i)$ as defined in (\ref{eqn:w_p(t)}) with Circulant matrix $C_n(t)$ is replaced by band Toeplitz matrices $B_n(t)$ with band length tends to infinity as n $\tends \infty$. They have not studied the process convergence of the process  $\{ w_p(t) ; t \geq 0\}$. The following theorem describes the process convergence of $\{ w_p(t) ; t \geq 0\}$ for fixed $p \geq 2$.
\begin{theorem}\label{thm:process}
	Suppose $ p\geq 2 $. Then as $n\to\infty$
	\begin{equation}
	\{ w_p(t) ; t \geq 0\} \stackrel{\mathcal D}{\rightarrow} \{N_p(t) ; t \geq 0\}, 
	\end{equation}
	where $\{N_p(t);  t \geq 0\}$  is defined as in Theorem \ref{thm:cirmulti}.	
\end{theorem}
\begin{remark}
In the above theorems we have considered the fluctuation of $w_p(t)$ for  $p\geq2$. For $p=0$, 
\begin{align*}
w_0(t) = \frac{1}{\sqrt{n}} \bigl\{ \Tr(I) - \E[\Tr(I)]\bigr\} = \frac{1}{\sqrt{n}}[n-n] = 0
\end{align*}
and hence it has no fluctuation.
For $p=1$,
\begin{align*}
w_1(t) = \frac{1}{\sqrt{n}} \bigl\{ \Tr(C_n(t)) - \E[\Tr(C_n(t))]\bigr\} = \frac{1}{\sqrt{n}} \big[n \frac{b_0(t)}{\sqrt n} - \E(n\frac{b_0(t)}{\sqrt n})\big] = b_0(t),
\end{align*}
as $\E(b_0(t))=0$. So $b_1(t)$ is distributed as $N(0,t)$ and its distribution does not depend on $n$. So we ignore theses two cases, for $p=0$ and $p=1$.  
\end{remark}

\begin{remark}\label{ft:w_p(t)}
For any $p \geq 2 $	and $t>0$, from Theorem \ref{thm:cirmulti} we get that
\begin{equation}\label{distributional convergence:w_p(t)}
 w_p(t) = \frac{1}{\sqrt{n}} \bigl\{ \Tr(C_n(t))^p - \E[\Tr(C_n(t))^p]\bigr\}   \stackrel{\mathcal D}{\rightarrow} N_p(t), 
\end{equation}	
where $N_p(t)$ has mean 0 and variance $\sigma_{p,t}^2 = (t)^pp!\ \sum_{s=0}^{p-1} f_p(s).$ This convergence of \eqref{distributional convergence:w_p(t)} is also follows from Theorem 1 of 
\cite{adhikari_saha2017}, where it was established in total variation norm. 
\end{remark}
 In Section \ref{sec:poly} we prove Theorem \ref{thm:circovar} and in Section \ref{sec:joint convergence} we prove Theorem \ref{thm:cirmulti}. We  use  method of moments and Cramer-Wold device to prove Theorem  \ref{thm:cirmulti}. In Section \ref{sec:process convergence} we use some results on process convergence and Theorem \ref{thm:cirmulti}, to prove Theorem \ref{thm:process} 

\section{Proof of Theorem \ref{thm:circovar}}\label{sec:poly}
We first define some notation which will be used   in the proof of Theorem \ref{thm:circovar} and Theorem \ref{thm:cirmulti}.
\begin{align}
A_p&=\{(i_1,\ldots,i_p)\in \Z^p\suchthat i_1+\cdots+i_p=0\;{(\mbox{mod $n$})},\; 0\le i_1,\ldots, i_p\le n-1\},\label{def:A_p}\\
A_{p}'&=\{(i_1,\ldots,i_p)\in \Z^p\suchthat i_1+\cdots+i_p=0\;{(\mbox{mod $n$})},\; 0\le i_1\neq i_2\neq \cdots\neq i_p\le n-1\},\nonumber\\
A_{p,s}&=\{(i_1,\ldots,i_p)\in \Z^p\suchthat i_1+\cdots+i_p=sn,\; 0\le i_1,\ldots, i_p\le n-1\},\nonumber \\
A_{p,s}'&=\{(i_1,\ldots,i_p)\in \Z^p\suchthat i_1+\cdots+i_p=sn,\; 0\le i_1\neq i_2\neq \cdots\neq i_p\le n-1\}.\nonumber 
\end{align}
The following result will be used in the proof of  Theorem \ref{thm:circovar}. For the proof the result, we refer the readers to \cite[Lemma 13]{adhikari_saha2017}.
\begin{result}\label{ft:variance}
Consider  $A_{p}$ as defined in \eqref{def:A_p}. Then 
\begin{align*}
\lim_{n\to \infty}\frac{|A_p|}{n^{p-1}}=\sum_{s=0}^{p-1}\lim_{n\to \infty}\frac{|A_{p,s}|}{n^{p-1}}=\sum_{s=0}^{p-1}f_p(s),
\end{align*}
where $f_p(s)=\frac{1}{(p-1)!}\sum_{k=0}^{s}(-1)^k\binom{p}{k}(s-k)^{p-1}.$
\end{result}
Note that to prove Theorem \ref{thm:circovar}, we have to deal with the traces of higher powers of  circulant matrices.  So here we calculate the trace of $(^p)$ for some positive integer $p$.
Let $e_1,\ldots,e_n$ be the standard unit vectors in $\mathbb R^n$, that is, $e_i=(0,\ldots,1,\ldots, 0)^t$ ($1$ in $i$-th place). Now for $i=1,\ldots, n$, we have 
\begin{align*}
(C_n)e_i=\mbox{$i$-th column}=\sum_{i_1=0}^{n-1}x_{i_1}e_{i-i_1 \mbox{ mod $n$}}.
\end{align*}
In the last equation $e_0$ stands for $e_n$. Now using the last equation, we get    
$$(C_n)^{2}e_i=\sum_{i_1,i_{2}=0}^{n-1}x_{i_1}x_{i_{2}}e_{i-i_1-i_2 \mbox{ mod $n$}},$$
and for $p>2$ 
\begin{align*}
(C_n)^{p}e_i&=\sum_{i_1,\ldots,i_{p}=0}^{n-1}x_{i_1}\cdots x_{i_{p}}e_{i-i_1-i_2-i_3\cdots -i_p \mbox{ mod $n$}}.
\end{align*}
Therefore the trace of $(C_n)^{p}$ can be written as 
\begin{align}\label{trace formula C_n}
\Tr[(C_n)^p]=\sum_{i=1}^{n}e_i^t(C_n)^p e_i=n\sum_{A_{p}}x_{i_1}\cdots x_{i_{p}},
\end{align}
where $A_p$ is as defined in \eqref{def:A_p}. For a similar result on the trace of band Toeplitz matrix see \cite{liu_wang2011}.

Now we are ready to prove Theorem \ref{thm:circovar}. We start with the following lemma.
\begin{lemma}\label{lem:w_p}
Suppose $0<t_1 \leq t_2$ and $p, q\geq 2$.
Then
\begin{align}\label{eqn:covariance}
\Cov\big(w_p(t_1),w_q(t_2)\big)&=\frac{1}{n^{\frac{p+q}{2}-1}}  \sum_ {r=0}^{q} C_{q,r} \sum_{A_p, A_q} \Big\{\E[u_{i_1}\cdots u_{i_p}{u_{j_1}} \cdots u_{j_r}v_{j_{r+1}}\cdots v_{j_q} ]\nonumber \\
&\qquad -\E[ u_{i_1}\cdots u_{i_p} ] \E[u_{j_1}\cdots u_{j_r}v_{j_{r+1}}\cdots v_{j_q}  ]        \Big\},
\end{align}
where
$u_i=b_i(t_1),\ v_i=b_i(t_2)- b_i(t_1)$ and $ \displaystyle {C_{q,r} = \binom{q}{r}}$.
\end{lemma}
\begin{proof}
Define
\begin{align*}
 U&=C_n(t_1)=  \frac{1}{\sqrt{n}} \Big(u_{j-i \;{(\mbox{mod $n$})}}\Big)_{i,j=1} ^ {n} \mbox{ and }\\
 V&=C_n(t_2)-C_n(t_1) =  \frac{1}{\sqrt{n}} \Big(v_{j-i \;{(\mbox{mod $n$})}}\Big)_{i,j=1} ^ {n}.
 \end{align*}
As $\E(w_p(t_1))=\E(w_q(t_2)=0$, we have 
\begin{align*}
 \Cov\big(w_p(t_1),w_q(t_2)\big)&=\E[w_p(t_1) w_q(t_2)]\\
 &= \frac{1}{n} \Big\{ \E[\Tr(C_n(t_1))^p\Tr(C_n(t_2))^q ]- \E[\Tr(C_n(t_1))^p]\E[\Tr(C_n(t_2))^q]   \Big\} \\
 &= \frac{1}{n} \Big\{ \E[\Tr(U)^p\Tr(U+V)^q ]- \E[\Tr(U)^p]\E[\Tr(U+V)^q]   \Big\}. 
 \end{align*}
 Using the trace formula (\ref{trace formula C_n}), we get
 \begin{align*}
 \E[\Tr(U^p)]&= \E\Big[n \sum_{A_{p}} \frac{u_{i_1}}{ \sqrt{n}} \cdots \frac{u_{i_p}}{ \sqrt{n}} \Big]
 = \frac{1}{ n^ {\frac{p}{2}-1}} \E\Big[ \sum_{A_{p}} u_{i_1} \cdots u_{i_p}\Big],\\
  \E[\Tr(U+V)^q]&= \E\Big[n \sum_{A_{q}} \frac{(u_{j_1} + v_{j_1})}{ \sqrt{n}} \cdots \frac{(u_{j_q} + v_{j_q})}{ \sqrt{n}} \Big] \\
  &= \frac{1}{ n^ {\frac{q}{2}-1}}\E\Big[ \sum_{A_{q}} (u_{j_1} + v_{j_1}) \cdots  (u_{j_q} + v_{j_q})\Big].
\end{align*}
 Therefore
\begin{align*}
\Cov\big(w_p(t_1),w_q(t_2)\big)&=  \frac{1}{ n^ {(\frac{p+q}{2}-1)}} \Bigg[ \E\Big\{ \Big( \sum_{A_{p}} u_{i_1} \cdots u_{i_p} \Big) \Big(   \sum_{A_{q}} (u_{j_1} + v_{j_1}) \cdots  (u_{j_q} + v_{j_q})   \Big)   \Big\} \\
 &\qquad  - \E \Big( \sum_{A_{p}} u_{i_1} \cdots u_{i_p} \Big) \Big(\E\sum_{A_{q}} (u_{j_1} + v_{j_1}) \cdots  (u_{j_q} + v_{j_q}) \Big)     \Bigg] \\
  &= \frac{1}{ n^ {(\frac{p+q}{2}-1)}} \Bigg[ \E\Big\{  \sum_{A_{p}, A_{q}} u_{i_1} \cdots u_{i_p} (u_{j_1} + v_{j_1}) \cdots  (u_{j_q} + v_{j_q})   \Big\} \\
  &\qquad  -   \sum_{A_{p} A_{q}}  \Big(\E(u_{i_1} \cdots u_{i_p})\E\big[ (u_{j_1} + v_{j_1}) \cdots  (u_{j_q} + v_{j_q})\big]\Big)     \Bigg] \\
  &= \frac{1}{ n^ {(\frac{p+q}{2}-1)}} \sum_{A_{p}, A_{q}} \Big[ \E\Big\{  u_{i_1} \cdots u_{i_p} (u_{j_1} + v_{j_1}) \cdots  (u_{j_q} + v_{j_q})   \Big\} \\
  &\qquad  -   \Big(\E(u_{i_1} \cdots u_{i_p})\E\big[ (u_{j_1} + v_{j_1}) \cdots  (u_{j_q} + v_{j_q})\big]\Big)     \Big].
\end{align*}
Hence
\begin{align*}
  \Cov\big(w_p(t_1),w_q(t_2)\big) &=\frac{1}{n^{\frac{p+q}{2}-1}}  \sum_ {r=0}^{q} C_{q,r} \sum_{A_p, A_q} \Big\{\E[u_{i_1} \cdots u_{i_p} u_{j_1} \cdots u_{j_r}v_{j_{r+1}} \cdots v_{j_q} ] \\
  &\qquad -\E[ u_{i_1} \cdots u_{i_p} ] \E[u_{j_1}\cdots u_{j_r}v_{j_{r+1}} \cdots v_{j_q}  ]        \Big\},
\end{align*}
where $C_{q,r}$ is a constant. It is easy to observe that $ \displaystyle {C_{q,r} = \binom{q}{r}}$. This completes the proof of the lemma.
\end{proof}

\begin{proof}[Proof of Theorem \ref{thm:circovar}]
 For $p,q\geq 2$ and $0 \leq r \leq q$, define
 \begin{align*}
 I_p&= (i_1, i_2, \ldots, i_p),\quad J_q=(j_1,j_2, \ldots, j_q),	\\
 J_r&= (j_1, j_2, \ldots, j_r),\quad J_{r+1,q}= (j_{r+1}, j_{r+2}, \ldots, j_q),\\
U_{I_p} &= u_{i_1}\cdots u_{i_p},\quad  U_{J_r} = u_{j_1}\cdots u_{j_r} \ \mbox{ and } \ V_{J_{r+1,q}} = v_{j_{r+1}}\cdots v_{j_q}.
\end{align*} 	
Therefore   \eqref{eqn:covariance} can be written as
  \begin{align*}
  \Cov\big(w_p(t_1),w_q(t_2)\big)= \frac{1}{n^{\frac{p+q}{2}-1}}  \sum_ {r=0}^{q} C_{q,r} \sum_{A_p, A_q} \Big\{\E[U_{I_p}U_{J_r}V_{J_{r+1,q}} ]    -\E[ U_{I_p}] \E[U_{J_r} V_{J_{r+1,q}}  ]        \Big\}, 
  \end{align*}
  where $U_{J_0}=1$ and $V_{J_{q+1,q}}=1$.
  Now for any fixed $r\in\{0,1, \ldots, q\}$,
 \begin{align} \label{equaton:T_1,T_2} 
  &\quad \frac{1}{n^{\frac{p+q}{2}-1}} \sum_{A_p, A_q} \Big\{\E[U_{I_p} U_{J_r} V_{J_{r+1,q}} ]    -\E[ U_{I_p}] \E[U_{J_r} V_{J_{r+1,q}}  ] \Big\} \nonumber \\
&=  \frac{1}{n^{\frac{p+q}{2}-1}} \sum_{A_p, A_q} \E[U_{I_p} U_{J_r} V_{J_{r+1,q}} ]    -  \frac{1}{n^{\frac{p+q}{2}-1}} \sum_{A_p} \E[ U_{I_p}] \sum_{ A_q}\E[U_{J_r} V_{J_{r+1,q}}  ] \nonumber\\
  &= T_1 - T_2, \mbox{ say}.
 \end{align}	
 Let us first calculate the second term $T_2$ of  \eqref{equaton:T_1,T_2}.
 Note that, for $ \E[U_{I_p}]$ to be non-zero, each random variable has to appear at least twice as the random variables have mean zero. Again the index variables satisfy one constraint because $(i_1,i_2, \ldots, i_p) \in A_p$. Thus we have at most $(\frac{p}{2} -1)$ free choices in the index set $A_p$.
 Hence
 \begin{equation} \label{eq:T_1 a}
 \frac{1}{n^{\frac{p}{2}-1}} \sum_{A_p} \E[ U_{I_p}] = O(1).
 \end{equation}
 Now for $r=0$, 
 $\E[U_{J_r} V_{J_{r+1,q}}]=\E[U_{J_0} V_{J_{1,q}}]=\E[v_{j_{1}}v_{j_2}\cdots v_{j_q}  ]$. By similar argument as above, we get 
 \begin{equation*}\label{eqn:r=0}
 \frac{1}{n^{\frac{q}{2}-1}}\E[U_{J_0} V_{J_{1,q}}]=O(1).
 \end{equation*}
Similarly for $r=q$, $\E[U_{J_r} V_{J_{r+1,q}}]=\E[U_{J_q} V_{J_{q+1,q}}]= \E[u_{j_1}u_{j_2}\cdots u_{j_q}]$ and hence
$$\frac{1}{n^{\frac{q}{2}-1}}\E[U_{J_q} V_{J_{q+1,q}}]=O(1).$$
Therefore for $r=0$ and $r=q$,
\begin{equation*}
T_2 = \frac{1}{n^{\frac{p+q}{2}-1}} \sum_{A_p} \E[ U_{I_p}] \sum_{A_q}\E[U_{J_r} V_{J_{r+1,q}} ] = o(1).
\end{equation*}	
For $0<r<q$, first observe that $ U_{J_r} $ and $ V_{J_{r+1,q}} $ are independent.  
Also note that for non-zero contribution from $\E[U_{J_r} V_{J_{r+1,q}}]$, no random variable can appear only once as random variables have mean zero. Therefore each indices in $  \{j_1,j_2, \ldots, j_r, j_{r+1},j_{r+2}, \ldots, j_q \} $ has to appear at least twice. Since $ U_{J_r} $ and $ V_{J_{r+1,q}} $ are independent therefore the maximum contribution comes when there is a self-matching in $\{ j_1,j_2, \ldots, j_r \} $ and also in $ \{ j_{r+1}, j_{r+2}, \ldots, j_q \}$. 
The index variables also satisfy one constraint because  $(j_1,j_2, \ldots, j_r, j_{r+1},j_{r+2}, \ldots, j_q ) \in A_q $. Therefore in such case, the maximum contribution is $  O( {n^{\frac{r}{2} + \frac{q-r}{2}-1}})  = O( {n^{\frac{q}{2}-1}}) $.
Therefore
\begin{equation} \label{eq:T_1 b}
\frac{1}{n^{\frac{q}{2}}} \sum_{A_q} \E[U_{J_r} V_{J_{r+1,q}}  ] = o(1).
\end{equation}
Therefore from (\ref{eq:T_1 a}) and (\ref{eq:T_1 b}), we have for $0<r<q$,
\begin{equation*}\label{eqn:1<r<q}  
 T_2 = \frac{1}{n^{\frac{p+q}{2}-1}} \sum_{A_p} \E[ U_{I_p}] \sum_{A_q} \E[U_{J_r} V_{J_{r+1,q}} ] = o(1). 
\end{equation*}
Hence, for $p,q\geq2$ and $0\leq r\leq q$,
\begin{equation}\label{eqn:T_2}
\lim_{n\to \infty} T_2 = \lim_{n\tends \infty} \frac{1}{n^{\frac{p+q}{2}-1}} \sum_{A_p, A_q} \E[ U_{I_p}] \E[U_{J_r} V_{J_{r+1,q}} ] = 0.
\end{equation}
 Let us now calculate the first term $T_1$ of  \eqref{equaton:T_1,T_2}. 
Without loss of generality, we assume $p \leq q$. Then there are following four cases: \\

\noindent \textbf{Case I: $p< q \mbox{ and } p\neq r$}.\\\\
We first assume $p<r$. 
Now suppose $\ell$ many indices from $I_p$ matches exactly with $\ell$ many entries from $J_r$ where $0\leq \ell \leq p$. Then a typical term in $T_1$ will be of the following form
\begin{equation*}
\E[u^2_{i_1} \cdots u^2_{i_\ell} u_{i_{\ell+1}} \cdots u_{i_p} u_{j_{\ell+1}} \cdots u_{j_r} v_{j_{r+1}} \cdots v_{j_q} ].
\end{equation*}
For non-zero contribution, each random variable has to appear at least twice and also there is a constraint,  $i_1 + i_2 + \cdots+ i_p = 0 (\mbox{mod $n$}$). Thus we have at most $( \ell + \frac{p-\ell}{2}-1)$ free choice in the index set $A_p$. Hence the maximum contribution due to $A_p$ will be $$ O(n^{\ell+ \frac{p-\ell}{2}-1}) = O(n^{\frac{p+\ell}{2}-1}).$$
Since $u_i$ and $v_i$ are independent, then we have at most $(\frac{r-\ell}{2}+ \frac{q-r}{2}-1) $ free choice in the index set $A_q $. Hence the maximum contribution due to $A_q$ will be
$$O( n^{\frac{r-\ell}{2}+ \frac{q-r}{2}-1})= O(n^{\frac{q-\ell}{2}-1}).$$
Therefore total contribution will be $ O (n^{\frac{p+q}{2}-2})$. 
Hence $T_1 \mathbb I_{\{p<r\}} = o(1).$ 
 
Now we assume $r \leq p$. If $\ell$ many indices from $I_p$ matches exactly with $\ell$ many entries from $J_r$, where $0\leq \ell \leq r$,  then a typical term in $T_1$ will be of the following form
 \begin{equation*}
 \E[u^2_{i_1} \cdots u^2_{i_\ell} u_{i_{\ell+1}} \cdots u_{i_p} u_{j_{\ell+1}} \cdots u_{j_r} v_{j_{r+1}} \cdots v_{j_q} ].
 \end{equation*}
Then by similar argument as above given for $p>r$,  the maximum contribution from $A_p$ will be $$ O(n^{\ell+ \frac{p-\ell}{2}-1}) = O(n^{\frac{p+\ell}{2}-1}),$$
 and the maximum contribution from $A_q$ will be
$$O( n^{\frac{r-\ell}{2}+ \frac{q-r}{2}-1})= O(n^{\frac{q-\ell}{2}-1}).$$
Therefore total contribution will be $ O (n^{\frac{p+q}{2}-2})$ and  hence
$T_1 \mathbb{I}_{\{r\leq p\}}= o(1).$ So 
$$T_1=T_1 \mathbb I_{\{p<r\}}+T_1 \mathbb{I}_{\{r\leq p\}}=o(1).$$

\noindent \textbf{Case II: $ p < q \mbox{ and } p = r$}.\\\\
 As $p=r$, $T_1$ will be
\begin{align*}
T_1 = \frac{1}{n^{\frac{p+q}{2}-1}} \sum_{A_p, A_q} \E[u_{i_1} \ldots u_{i_p} u_{j_1} \ldots u_{j_p} v_{j_{p+1}} \ldots v_{j_q} ],
\end{align*}
Due to $A_p$, the maximum contribution  will be obtained when $\{i_1, i_2, \ldots , i_p\}$ exactly matchs with $\{j_1, j_2, \ldots , j_p\}$ and each $(i_1, i_2, \ldots , i_p)$ consist of distinct elements. Thus maximum contribution due to $A_p$ is $O(n^{p-1})$. The maximum contribution due to $A_q$ is $ O(n^{(\frac{q-p}{2}-1)})$.
Therefore total contribution will be $O(n^{\frac{p+q}{2}-2})$. Hence
$T_1 = o(1).$ \\

\noindent \textbf{Case III: $ p = q \mbox{ and }  r<p$}. \\\\
In this case  $T_1 $ will be
\begin{equation*}
 T_1 = \frac{1}{n^{p-1}} \sum_{A_p, A_p} \E[u_{i_1} \cdots u_{i_p} u_{j_1} \cdots u_{j_r} v_{j_{r+1}} \cdots v_{j_p} ].
 \end{equation*}
Suppose $\ell$ many indices from $I_p$ matches exactly with $\ell$ many indices from $J_r$ where $0\leq \ell \leq r$. Then a typical term in $T_1$ will be of the following form
\begin{equation*}
\E[u^2_{i_1} \cdots u^2_{i_\ell} u_{i_{\ell+1}} \cdots u_{i_p} u_{j_{\ell+1}} \cdots u_{j_r} v_{j_{r+1}} \cdots v_{j_p} ].
\end{equation*} 
By similar calculation as done in Case I, the maximum contribution due to $A_p$ will be 
$ O(n^{\ell+ \frac{p-\ell}{2}-1}) = O(n^{\frac{p+\ell}{2}-1})$. 
The maximum contribution due to $A_{q(=p)}$ will be 
$ O( n^{\frac{r-\ell}{2}+ \frac{p-r}{2}-1})= O(n^{\frac{p-\ell}{2}-1})$. 
Therefore total contribution will be $O(n^{p-2})$. Hence
$T_1 = o(1).$ \\

\noindent \textbf{Case IV: $ p = q=r$}. \\\\
In this case $T_1 $ will be
\begin{equation*}
T_1 = \frac{1}{n^{p-1}} \sum_{A_p, A_{p}} \E[u_{i_1} \cdots u_{i_p} u_{j_1} \cdots u_{j_p} ].
\end{equation*}
Maximum contribution will come when $\{i_1, i_2, \dots, i_p\}$ completely matches with $\{j_1, j_2, \ldots, j_p\}$ and each entries of $(i_1,i_2, \ldots, i_p)$ are distinct and contribution will be $O(n^{p-1}).$ Therefore we have 

\begin{equation*}
T_1 = \frac{1}{n^{p-1}} p! \sum_{A_p} \E[u^2_{i_1} \cdots u^2_{i_p} ].
\end{equation*}
The factor $p!$ is coming because $\{i_1, i_2, \ldots, i_p\}$ can match with the given vector $(j_1,j_2, \ldots, j_p)$ in $p!$ ways. Hence
\begin{equation*}
T_1 = \frac{1}{n^{p-1}} p! \sum_{A_p} \E[u^2_{i_1} \cdots u^2_{i_p}] = \frac{p!}{n^{p-1}} \sum_{s=0}^{p-1} \sum_{A_{p,s}} \E[u^2_{i_1} \cdots u^2_{i_p}],
\end{equation*}
where $A_{p,s}$ is as defined in  (\ref{def:A_p}). Now
\begin{align*}
 \lim_{n\to \infty} T_1 &= \lim_{n\to \infty} \frac{p!}{n^{p-1}} \sum_{s=0}^{p-1} \sum_{A_{p,s}} \E[u^2_{i_1} \cdots u^2_{i_p}]\\
 &= \lim_{n\to \infty} \frac{p!}{n^{p-1}} \sum_{s=0}^{p-1} \sum_{A'_{p,s}} \E[u^2_{i_1} \cdots u^2_{i_p}] \\
 &= (t_1)^p p! \sum_{s=0}^{p-1} \lim_{n\to \infty} \frac{|A'_{p,s}|}{n^{p-1}} \\
 &= (t_1)^p p! \sum_{s=0}^{p-1} \lim_{n\to \infty} \frac{|A_{p,s}|}{n^{p-1}},
\end{align*}
where $A_{p,s}$ and $A'_{p,s}$ are defined in  (\ref{def:A_p}). Also note that the last equality is coming because if any two indices of $\{i_1, i_2, \ldots, i_p \}$ are equal then $|A_{p,s}| - |A'_{p,s}| = O(n^{p-2})$, which gives zero contribution in limit. Therefore from Result \ref{ft:variance}, 
 $$ \lim_{n\to \infty}T_1 = (t_1)^p p! \sum_{s=0}^{p-1}f_p(s).$$ 
Hence combining all four cases we have that if $p=q=r,$ then 
\begin{equation}\label{eqn:T_1}
\lim_{n\to \infty} T_1 = \lim_{n\to \infty} \frac{1}{n^{\frac{p+q}{2}-1}} \sum_{A_p, A_q} \E[U_{I_p} U_{J_r} V_{J_{r+1,q}} ] = (t_1)^p p! \sum_{s=0}^{p-1}f_p(s). \\
\end{equation}
Therefore from  (\ref{eqn:T_1}) and  (\ref{eqn:T_2}), for $0\leq r \leq q$ we have
\begin{align*}
&\quad \lim_{n\to \infty} \frac{1}{n^{\frac{p+q}{2}-1}} \sum_{A_p A_q}  \Big\{\E[U_{I_p}U_{J_r}V_{J_{r+1,q}} ]    -\E[ U_{I_p}]             \E[U_{J_r} V_{J_{r+1,q}}  ] \Big\} \\
&= \left\{\begin{array}{ccc} 	 
 (t_1)^pp!\ \sum_{s=0}^{p-1}	f_p(s) & \text{if} &p=q=r,\\
 0 & \text{otherwise}.&  	 	 
 \end{array}\right.
\end{align*} 
Now for $q=r,$ we have $ \ C_{q,r} = 1$, and hence
\begin{equation*}
\lim_{n\to \infty} \Cov\big(w_p(t_1),w_q(t_2)\big)
= \left\{\begin{array}{ccc} 	 
(t_1)^pp!\ \sum_{s=0}^{p-1}	f_p(s) & \text{if} & p=q,	\\\\
0 & \text{if} & p\neq q.  	 
\end{array}\right.
\end{equation*}

This completes the proof of Theorem \ref{thm:circovar}.
\end{proof}
The following result is a consequence of  Theorem \ref{thm:circovar}. It will be used in the proof of Theorem \ref{thm:cirmulti}.
\begin{corollary}\label{cor:existence of Gaussian process} 
For $p\geq 2$, there exists a centred Gaussian process $\{N_p(t);t\geq 0\}$  such that
$$\Cov(N_p(t_1),N_p(t_2))= (\min\{t_1,t_2\})^pp!\ \sum_{s=0}^{p-1}	f_p(s),$$
where $f_p(s)$ as in \eqref{eqn:f_p(s)}. 
\end{corollary}
\begin{proof}
Define $K_n:I\times I\to \mathbb R$ as
$K_n(t_1,t_2)=\Cov\big(w_p(t_1),w_p(t_2)\big)$, where $I=[0,\infty)$.  $K_n$ is a covariance kernal and  therefore $K_n$ is positive definite in the sense that $\det(K_n(t_i,t_j))\geq 0$ for any $t_1,t_2,\ldots,t_n\in I$. Now define
$K(t_1,t_2):=\lim_{n\to \infty}K_n(t_1,t_2)$ for $(t_1,t_2)\in I\times I$. Then $K$ is symmetric and positive definite, as $K_n$ is so. Hence $K$ is a covariance kernel. As the projection of $n$ dimensional Gaussian distribution with mean zero and covariance matrix $(K(t_i,t_j))_{1\leq i,j\leq n}$ to the first $(n-1)$ co-ordinates is the $(n-1)$ dimensional Gaussian distribution with mean zero and covariance matrix $(K(t_i,t_j))_{1\leq i,j\leq n-1}$, by Kolmogorov consistency theorem  there exists a centred Gaussian process $\{N_p(t);t\geq 0\}$ with covariance kernel $K$. This completes the proof.
\end{proof}

\section{Proof of Theorem \ref{thm:cirmulti}}\label{sec:joint convergence} 
We use method of moments to prove Theorem \ref{thm:cirmulti}. 
Here  we begin with  some notation and definitions. First recall $A_p$ from \eqref{def:A_p} in Section \ref{sec:poly},
$$A_p=\{(i_1,\ldots,i_p)\in \Z^p\suchthat i_1+\cdots+i_p=0\;{(\mbox{mod $n$})},\; 0\le i_1,\ldots, i_p\le n-1\}.$$ 
For a  vector  $ J = (j_1, j_2, \ldots, j_p)\in A_p, \ p\geq 2$, we define a multi set $S_J$ as 
\begin{equation}\label{def:S_j}
S_{J} = \{j_1, j_2, \ldots,j_p\}.
\end{equation}
\begin{definition}\label{def:connected}
	Two vectors $J =(j_1, j_2, \ldots, j_p)$ and $J' = (j'_1, j'_2, \ldots, j'_q)$, where $J \in A_p$ and $J' \in A_q$, are said to be \textit{connected} if
	$S_{J}\cap S_{J'} \neq \emptyset$.	
\end{definition}
For $1 \leq i \leq \ell$, suppose $J_i \in A_{p_i}$. Now, we define cross-matched and self-matched element in $\displaystyle{\cup_{i=1}^{\ell} S_{J_i} }$.
\begin{definition}\label{def:cross matched}
	An element in $\displaystyle{\cup_{i=1}^{\ell} S_{J_i} }$ is called \textit{cross-matched} if it appears at least in two distinct $S_{J_i}$. If it appears in $k$ many ${S_{J_i}}^,s$, then we say its \textit{cross-multiplicity} is $k$.
\end{definition}
\begin{definition}\label{def:self-matched}
	An element in $\displaystyle{\cup_{i=1}^{\ell} S_{J_i} }$ is called \textit{self-matched} if it appears more than one in one of $S_{J_i}$. If it appears $k$ many times in $S_{J_i}$, then we say its \textit{self-multiplicity} in $S_{J_i}$ is $k$.
	
	An element in $\displaystyle{\cup_{i=1}^{\ell} S_{J_i} }$  can be both self-matched and cross-matched.
\begin{definition}\label{def:cluster}
	Given a set of vectors $S= \{J_1, J_2, \ldots, J_\ell \}$, where $J_i \in A_{p_i}$ for $1 \leq i \leq \ell$, a subset $T=\{J_{n_1}, J_{n_2}, \ldots, J_{n_k}\}$ of $S$ is called a \textit{cluster} if it satisfies the following two conditions: 
	\begin{enumerate}
	\item[(i)] For any pair $J_{n_i}, J_{n_j}$ from  $T$ one can find a chain of vectors from 
	$T$, which starts with $J_{n_i}$ and ends with $J_{n_j}$ such that any two neighbouring vectors in the chain are connected.
	\item[(ii)] The subset $\{J_{n_1}, J_{n_2}, \ldots, J_{n_k}\}$ can not  be enlarged to a subset which preserves condition (i).
	\end{enumerate}
\end{definition}
\begin{definition} The number of vectors in a cluster is called the \textit{length} of the cluster.
\end{definition}
Let us understand the notion of a cluster using a graph. We denote the vectors from ${A_{p_i}}^,s$ by vertices  and the connection between two vectors from ${A_{p_i}}^,s$ by an edge. Then the clusters are nothing but the connected components in that graph. For example, suppose $\{J_1,J_2,J_3\}$ form a cluster in $\{J_1,J_2,J_3,J_4,J_5\}$. Then in the graph, $\{J_1,J_2,J_3\}$ form a connected component (see figure \ref{fig:cluster1}). If $\{J_4,J_5\}$  form a different cluster then the graph will have two connected components (see Figure \ref{fig:cluster2}). 

\begin{figure}[h]
\centering
\includegraphics[height=40mm, width =90mm ]{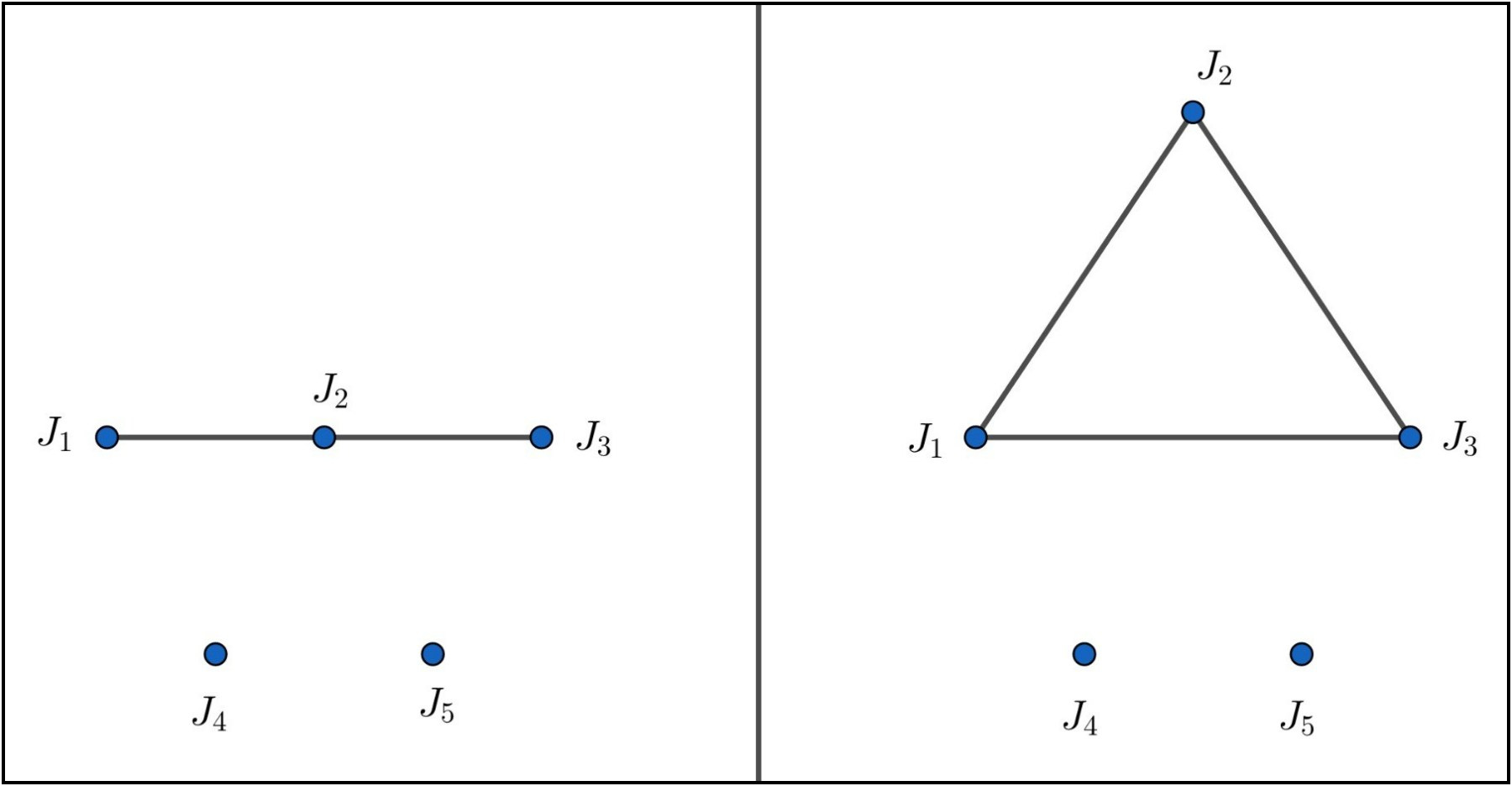}
\caption{$\{J_1,J_2,J_3\}$ form a cluster in $\{J_1,J_2,J_3,J_4,J_5\}$, and $J_4$ and $J_5$ are not connected to anyone.}\label{fig:cluster1}
\end{figure}


\begin{figure}[h]
\centering
\includegraphics[height=40mm, width =90mm ]{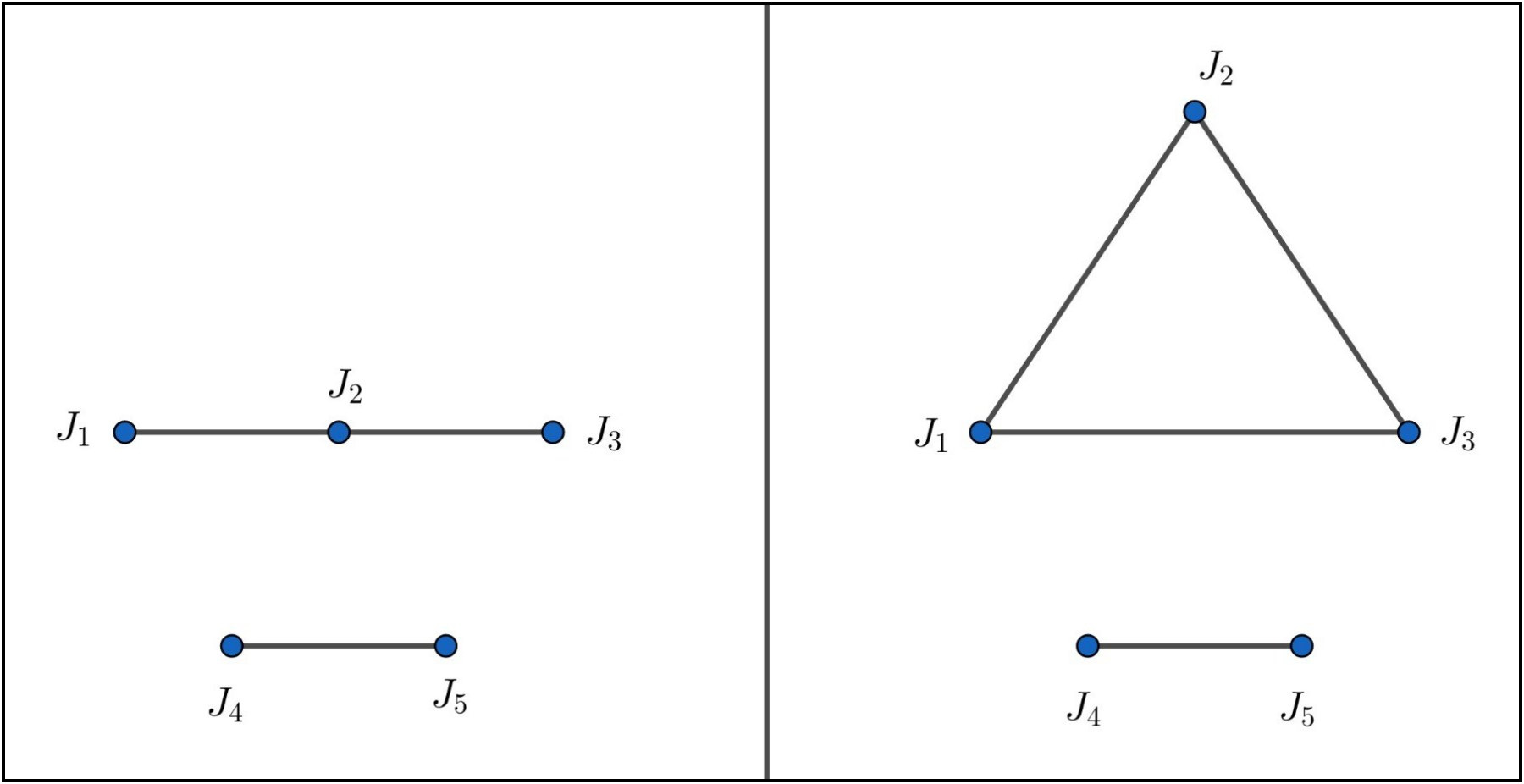}
\caption{In both the figures there are two connected components, one consists of $\{J_1,J_2,J_3\}$ and another consists of $\{J_4,J_5\}$.}\label{fig:cluster2}
\end{figure}



Now we define a subset $B_{P_\ell}$ of the Cartesian product  $ A_{p_1} \times A_{p_2} \times \cdots \times A_{p_\ell}$ where $p_i\geq 2$ and $A_{p_i}$ is as defined in \eqref{def:A_p}.  
\begin{definition}\label{def:B_{P_l}}
 Let $\ell \geq 2$ and  $P_\ell = (p_1,p_2, \ldots, p_\ell ) $ where $ p_i \geq2$. Now $ B_{P_\ell}$ is a subset of $ A_{p_1} \times A_{p_2} \times \cdots \times A_{p_\ell}$ such that  $ (J_1, J_2, \ldots, J_\ell) \in B_{P_\ell} $ if 
 \begin{enumerate} 
 	\item[(i)] $J_1, J_2, \ldots, J_\ell $ form a cluster,
 	\item[(ii)] each element in  $\displaystyle{\cup_{i=1}^{\ell} S_{J_i} }$ has  multiplicity greater equal to two. 
 \end{enumerate}
\end{definition}

\end{definition}
The set $B_{P_\ell}$ will play an important role in the proof of Theorem \ref{thm:cirmulti}. The next lemma gives us the cardinality of $B_{P_\ell}$.
\begin{lemma}\label{lem:cluster}
For  $\ell \geq 3 $, 
\begin{equation}\label{equation:cluster}
|B_{P_\ell }| = O \big(n^{\frac{p_1+p_2 + \cdots + p_\ell}{2}-\ell}\big).
\end{equation}
\end{lemma}
\begin{remark}
The above lemma is not true if $\ell=2$ and $p_1= p_2$. Suppose $(J_1,J_2)\in B_{P_2}$. Then all $p_1$ entries of $J_1$ may coincide  with $p_2(=p_1)$ many entries of $J_2$ and hence 
$$|B_{P_2}|=O(n^{p_1-1}).$$
In this situation $|B_{P_2}|>O(n^{\frac{p_1+p_2}{2}-2})$.
\end{remark}
\begin{proof}
Let us first look at the proof for $\ell=3$. Suppose $(J_1, J_2, J_3)\in B_{P_3}$, then $(J_1, J_2, J_3)$   can be connected in the following two ways:\\\\
\noindent \textbf{Type I:} \textit{$J_1$ is connected with $J_2$, $J_2$ is connected with $J_3$, but $J_3$ is not  connected with $J_1$.} 

Suppose $r_1$ many entries of $J_1$ coincide with $r_1$ many entries of $J_2$;
$r_2$ many entries of $J_2$ coincide with $r_2$ many entries of $J_3$ where $0<r_1<p_1$, $0<r_2<p_3$ and $r_1+r_2 < p_2$, and there is no common entry between $J_1$ and $J_3$.  
\begin{figure}[h]
\centering\vskip-5pt
\includegraphics[height=20mm, width =70mm ]{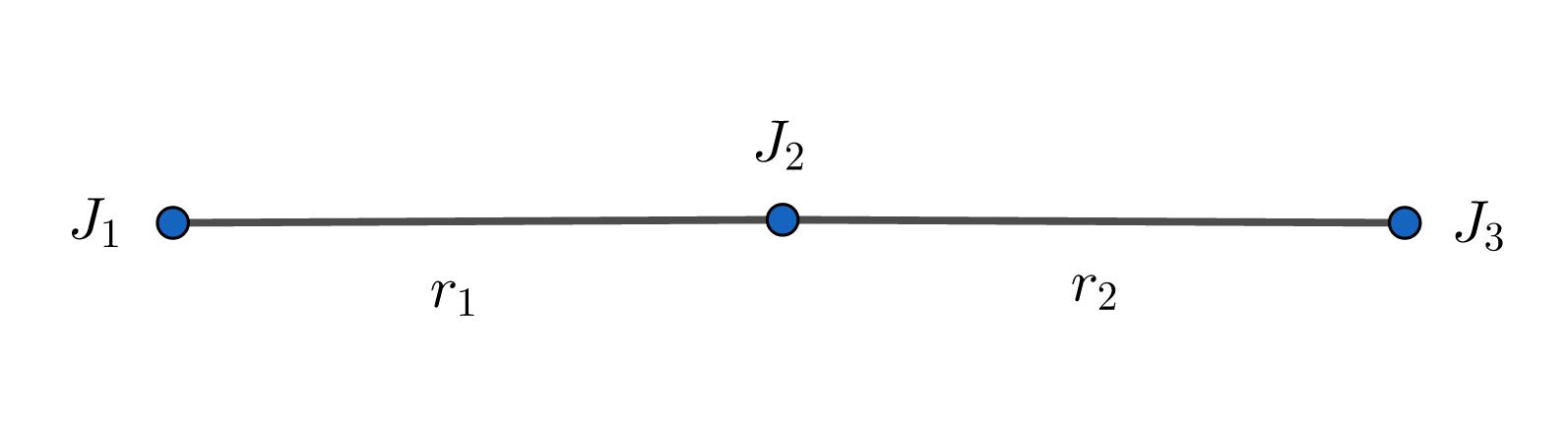}
\vskip-10pt
\caption{Connection between $J_1,J_2$ and $J_3$: Type I.}	
\end{figure} 
As each entry in $S_{J_1}\cup S_{J_2}$ has multiplicity at least two,
the $(p_1-r_1)$ many entries in $J_1$ have to be pair matched among themselves. So the contribution from $J_1$ will be $O\big(n^{(r_1+ \frac{p_1-r_1}{2}-1)}\big)$. Now after fixing entries in $J_1$, by similar argument the maximum contribution from $J_2$ will be $O\big(n^{(r_2 + \frac{p_2-r_1-r_2}{2}-1)}\big)$. After fixing $J_1$ and $J_2$, contribution from $J_3$ will be  $O\big(n^{( \frac{p_3-r_2}{2}-1)}\big)$. So  the cardinality of $B_{P_3}$  will be  
$$ O\big(n^{(r_1+ \frac{p_1-r_1}{2}-1) + (r_2 + \frac{p_2-r_1-r_2}{2}-1) +( \frac{p_3-r_2}{2}-1) }\big) 
 = O\big(n^{ \frac{p_1+p_2+p_3}{2}-3}\big).$$ 

 \noindent \textbf{Type II:} \textit{All three vectors $J_1$, $J_2$, $J_3$ are connected with each other.}
 
Suppose $r$ many entries of $J_1$ matches with  $r$ many entries of $J_2$ and $J_3$ both; $r_1$ many entries of $J_1$ coincide only with $r_1$ many entries from $J_2$; $r_2$ many entries of $J_2$ coincide only with $r_2$ many entries of $J_3$; $r_3$ many entries of $J_3$ coincide only with $r_3$ many entries of $J_1$, where
 $$r, r_1, r_2, r_3 \geq  0 \ , \ r+r_1 + r_2 \leq  p_2,\ r+r_1 + r_3\leq p_1, \ r+r_2+r_3\leq p_3. $$
\begin{figure}[h]
\centering\vskip-10pt
\includegraphics[height=40mm, width =50mm ]{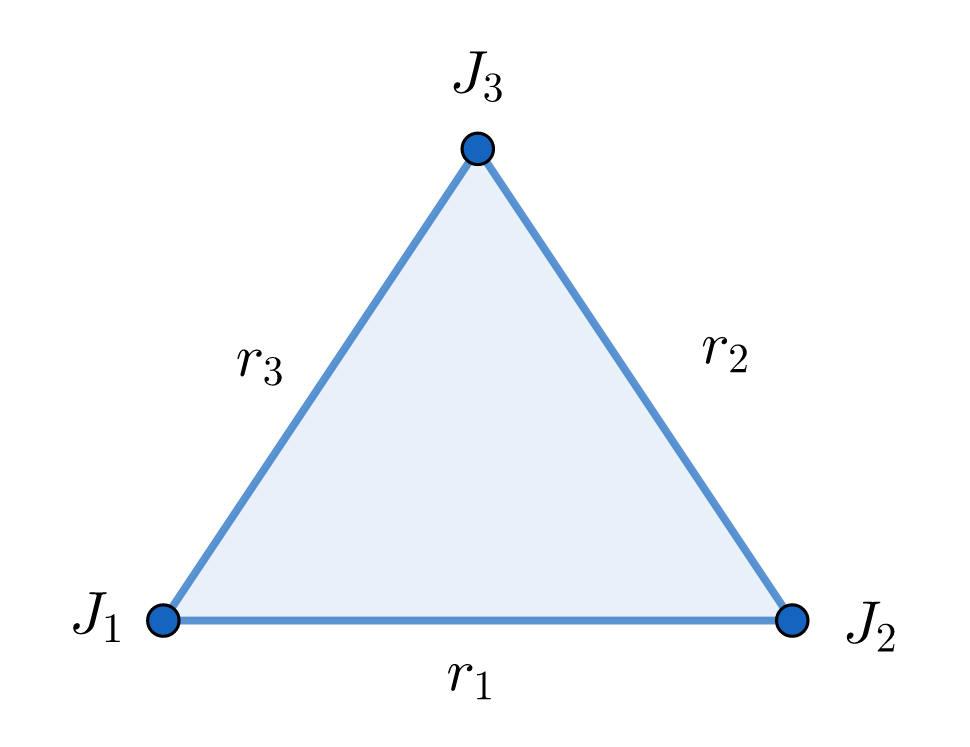}
\vskip-10pt
\caption{Connection between $J_1,J_2$ and $J_3$: Type II.}	
\end{figure} 
Like in Type I, we first fix the entries of $J_1$, then $J_2$ and then $J_3$. Therefore in this case also the cardinality  of $B_{P_3}$  will be of the order of or bounded by
\begin{align} \label{eqn:O(r)}
O & \big(n^{(r+r_1+r_3+ \frac{p_1-r-r_1-r_3}{2}-1) + (r_2 + \frac{p_2-r-r_1-r_2}{2}-1) +( \frac{p_3-r-r_2-r_3}{2}-1) }\big) \nonumber \\
& = O\big(n^{ \frac{p_1+p_2+p_3}{2}-\frac{r}{2}-3}\big).
\end{align}
Observe that, in this case the cardinality of $B_{P_3}$ will be $O \big(n^{\frac{p_1+p_2 + p_3}{2}-3}\big)$ only if $r=0$.
Therefore combining both the cases, Type I and Type II, we have
$$|B_{P_3 }| = O \big(n^{\frac{p_1+p_2 + p_3}{2}-3}\big).$$
So Lemma \ref{lem:cluster} is true for $\ell=3$.

 Also observe that in (\ref{eqn:O(r)}), $r>0$ if and only if there exists an element in  $\displaystyle{\cup_{i=1}^{3} S_{J_i} }$ which has cross-multiplicity greater than two. For $\ell \geq 3$ case, we calculate cardinality of $B_{P_\ell }$ when each element of $\displaystyle{\cup_{i=1}^{\ell} S_{J_i} }$ has cross-multiplicity less equal to two. Because if an element has cross-multiplicity greater than two, then $r>0$ and such cases have lesser contribution (in order of $n$) to the cardinality of $B_{P_\ell }$, as we have observed in $\ell =3$ case.  
 
Now we prove the lemma \ref{lem:cluster}, for $\ell$ vectors. Note that we consider that the cross multiplicity of every element of  $\displaystyle{\cup_{i=1}^{\ell} S_{J_i} }$ is less equal to two. Suppose 
$r_{1,i}$ many entries of $J_1$ coincide with 
$r_{1,i}$ many entries of 
$J_i$ for $2 \leq i \leq \ell$, where $r_{1,i} \geq 0$ for $2 \leq i \leq \ell$ and $r_{1,2}+ r_{1,3}+ \cdots + r_{1,\ell} <p_1.$
 Since $\{J_1, J_2, \ldots, J_\ell \}$ form a cluster, there is $i \in \{ 2, 3, \ldots , \ell\}$ such that $r_{1,i}>0$.
 
 Suppose 
 $r_{2,i}$ many entries of $J_2$ coincide with $r_{2,i}$ many entries of $J_i$ for $3 \leq i \leq \ell$, where $r_{2,i} \geq 0$ for $3 \leq i \leq \ell$. Also note that from the above consideration on $J_1$, $r_{1,2}$ many entries of $J_1$ coincide with $r_{1,2}$ many entries of $J_2$ and hence $r_{1,2}+ r_{2,3}+ r_{2,4}+ \cdots + r_{2,\ell} <p_2.$ 
Since $\{J_1, J_2, \ldots, J_\ell \}$ form a cluster, there is $i \in \{ 3, 4, \ldots , \ell\}$ such that $r_{2,i}>0$ or $r_{1,2}>0$.
 
 Similarly for $1\leq s \leq \ell -1$, suppose 
 $r_{s,i}$ many entries of $J_s$ coincide with $r_{s,i}$ many entries of $J_{i}$ for $s+1 \leq i \leq \ell$, where $r_{s,i} \geq 0$ for $s+1 \leq i \leq \ell$. Also note that from the above considerations on $J_1, J_2, \ldots, J_{s-1}$; $r_{i,s}$ many entries of $J_s$ coincide with $r_{i,s}$ many entries of $J_{i}$, for $1 \leq i \leq s-1$, where $r_{i,s} \geq 0$ for $1 \leq i \leq s-1$ and hence $r_{1,s}+ r_{2,s}+ \cdots + r_{s-1, s} + r_{s, s+1}+ r_{s, s+2}+ \cdots + r_{s, \ell} <p_s.$
 Since $\{J_1, J_2, \ldots, J_\ell \}$ form a cluster, there is $i \in \{ 1, 2, \ldots , s-1\}$ such that $r_{i, s}>0$ or there is $i \in \{ s+1, s+2, \ldots , \ell\}$ such that $r_{s, i}>0$
 
 From the above considerations on $J_1, J_2, \ldots, J_{\ell-1}$, note that $r_{i, \ell}$ many entries of $J_\ell$ coincide with  $r_{i, \ell}$ of $J_i$ for $1 \leq i \leq  \ell -1$, where $r_{i \ell} \geq 0$ for $1 \leq i \leq  \ell -1$ and $r_{1, \ell} + r_{2, \ell}+ \cdots + r_{\ell -1, \ell} < p_\ell.$
 Since $\{J_1, J_2, \ldots, J_\ell \}$ form a cluster, there is $i \in \{ 1, 2, \ldots , \ell-1\}$ such that $r_{i, \ell}>0$
 
 Now we calculate the contributions of ${J_i}^,s$, as we have calculated for the $\ell = 3$ case. The maximum contribution from $J_1$ will be
 \begin{equation} \label{eqn:OP_1}
 O\big( n^{ ( r_{1,2}+ r_{1,3}+ \cdots + r_{1,\ell} + \frac{p_1 - (r_{1,2}+ r_{1,3}+ \cdots + r_{1,\ell})}{2} -1) } \big).
 \end{equation}
 Now after fixing entries in $J_1$, the maximum contribution from $J_2$ will be
 \begin{equation} \label{eqn:OP_2}
 O\big( n^{ (r_{2,3}+ r_{2,4} + \cdots + r_{2,\ell} + \frac{p_2 - (r_{1,2}+ r_{2,3}+ r_{2,4}+ \cdots + r_{2,\ell})}{2} -1)} \big).
 \end{equation}
 Now similarly, after fixing entries in $J_1, J_2, \ldots , J_{s-1}$, the maximum contribution from $J_s$ will be
 \begin{equation} \label{eqn:OP_x}
 O\big(  n^{ (r_{s, s+1}+ r_{s, s+2}+ \cdots + r_{s, \ell} 
 	 + \frac{ p_s - (r_{1, s}+ r_{2, s}+ \cdots + r_{s-1, s} + r_{s, s+1}+ r_{s, s+2}+ \cdots + r_{s, \ell} ) } {2} -1 ) } \big).
 \end{equation}
 Now finally, after fixing entries in $J_1, J_2, \ldots , J_{\ell-1}$, the maximum contribution from $J_\ell$ will be
 \begin{equation} \label{eqn:OP_last}
 O\big( n^{( \frac{ p_\ell - (r_{1, \ell} + r_{2, \ell}+ \cdots + r_{\ell -1, \ell}) }{2} -1  )} \big).
 \end{equation}
 Therefore from $(\ref{eqn:OP_1})$ - $(\ref{eqn:OP_last})$, the cardinality of $B_{P_\ell}$ will be 
 \begin{align*}
 O\big(  &  n^{ ( r_{1,2}+ r_{1,3}+ \cdots + r_{1,\ell} + \frac{p_1 - (r_{1,2}+ r_{1,3}+ \cdots + r_{1,\ell})}{2} -1) } + \\
 &  \ \ \ \ n^{ (r_{2,3}+ r_{2,4} + \cdots + r_{2,\ell} + \frac{p_2 - (r_{1,2}+ r_{2,3}+ r_{2,4}+ \cdots + r_{2,\ell})}{2} -1)} + \\
 & \ \ \ \  \cdots +  n^{ (r_{s, s+1}+ r_{s, s+2}+ \cdots + r_{s, \ell} 
 	+ \frac{ p_s - (r_{1, s}+ r_{2, s}+ \cdots + r_{s-1, s} + r_{s, s+1}+ r_{s, s+2}+ \cdots + r_{s, \ell} ) } {2} -1 ) } + \\
 & \ \ \ \  \cdots + n^{( \frac{ p_\ell - (r_{1, \ell} + r_{2, \ell}+ \cdots + r_{\ell -1, \ell}) }{2} -1  )} \big) \\
 & \ \ \ \  = O \big(n^{\frac{p_1+p_2 + \cdots + p_\ell}{2}-\ell}\big).
 \end{align*}
Hence
$$|B_{P_\ell }| = O \big(n^{\frac{p_1+p_2 + \cdots + p_\ell}{2}-\ell}\big).$$
This complete the proof of Lemma \ref{lem:cluster}.

\end{proof}
The following lemma is an easy consequence of Lemma \ref{lem:cluster}.
\begin{lemma}\label{lem:maincluster}
Suppose $\{J_1, J_2, \ldots, J_\ell \} $ form a cluster where $J_i\in A_{p_i}$ with $p_i\geq 2$ for $1\leq i\leq \ell$. Then for $\ell \geq 3,$
\begin{equation}\label{equation:maincluster}
	\frac{1}{ n^{\frac{p_1+p_2+ \cdots + p_\ell - \ell}{2}} } \sum_{A_{p_1}, A_{p_2}, \ldots, A_{p_\ell}} \E\Big[\prod_{k=1}^{\ell}\Big(b_{J_k}(t_k) - \E(b_{J_k}(t_k))\Big)\Big] = o(1),
\end{equation}
where $0<t_1 \leq t_2 \leq \cdots \leq t_\ell $ and  for $k \in \{1,2, \ldots, \ell \} $, 
$$J_k = (j_{k,1}, j_{k,2}, \ldots, j_{k,p_k} ) \ \mbox{and} \ b_{J_k}(t_k) = b_{j_{k,1}}(t_k) b_{j_{k,2}}(t_k) \cdots b_{j_{k,p_k}}(t_k).$$	
\end{lemma}
\begin{proof} First observe that   $\E\Big[\prod_{k=1}^{\ell}\Big(b_{J_k}(t_k) - \E(b_{J_k}(t_k))\Big)\Big]$ will be non-zero only if each $b_i(t)$ appears at least twice in the collection $\{b_{j_{k,1}}(t_k), b_{j_{k,2}}(t_k), \ldots ,b_{j_{k,p_k}}(t_k); 1\leq k\leq \ell\}$, because $\E(b_i(t))=0$ for $t\geq 0$. Therefore
\begin{equation}\label{eqn:equality_reduction}
\sum_{A_{p_1}, \ldots, A_{p_\ell}} \hspace{-3pt}\E\Big[\prod_{k=1}^{\ell}\Big(b_{J_k}(t_k) - \E(b_{J_k}(t_k))\Big)\Big]=\sum_{(J_1,\ldots,J_\ell)\in B_{P_\ell}} \hspace{-3pt} \E\Big[\prod_{k=1}^{\ell}\Big(b_{J_k}(t_k) - \E(b_{J_k}(t_k))\Big)\Big],
\end{equation} 
where $B_{P_\ell}$ as in Definition \ref{def:B_{P_l}}. Also note that  for $0<t_1 \leq t_2 \leq \cdots \leq t_\ell $ and $m\in \mathbb N$,  
\begin{equation}\label{eqn:higher moment finite}
\sup_{1\leq j\leq \ell}\E|b_i(t_j)|^{2m}=\frac{(2m)!}{2^m m!}t_\ell^m,
\end{equation}  
as $\{b_i(t);t\geq 0\}_{i\geq 0}$ are independent standard Brownian motions. Therefore for fixed $0<t_1 \leq t_2 \leq \cdots \leq t_\ell $ and $p_1,p_2,\ldots,p_\ell \geq 2$, there exists $\alpha_\ell>0$, which depends only on $t_\ell$ and $p_1,p_2,\ldots,p_\ell$, such that  
\begin{equation}\label{eqn:modulus finite}
\Big|\E\big[\prod_{k=1}^{\ell}\big(b_{J_k}(t_k) - \E(b_{J_k}(t_k))\big)\big]\Big|\leq \alpha_\ell
\end{equation} 
for all $(J_1, J_2, \ldots, J_\ell)\in A_{p_1}\times A_{p_2}\times \cdots \times A_{p_\ell}$.
 
Now using \eqref{eqn:equality_reduction} and \eqref{eqn:modulus finite}, we have
\begin{align*}
	 \sum_{A_{p_1}, A_{p_2}, \ldots, A_{p_\ell}} \Big|\E\big[\prod_{k=1}^{\ell}\big(b_{J_k}(t_k) - \E(b_{J_k}(t_k))\big)\big]\Big|
 \leq  \sum_{(J_1,J_2,\ldots,J_\ell)\in B_{P_\ell}} \alpha_{\ell} 
\ = |B_{p_\ell}| \ \alpha_\ell.
\end{align*}
Now \eqref{equation:maincluster} follows from Lemma \ref{lem:cluster}. This completes the proof. 
\end{proof}

We shall use the above lemmata to prove Theorem \ref{thm:cirmulti}.

\begin{proof}[Proof of Theorem \ref{thm:cirmulti}] We use Cram\'er-Wold device to prove Theorem \ref{thm:cirmulti}. So it is enough to show that, for $0<t_1 \leq t_2 \leq \cdots \leq t_\ell $ and $p_1, p_2, \ldots , p_\ell \geq 2$,
$$\lim_{n\to\infty}\E[w_{p_1}(t_1)w_{p_2}(t_2) \cdots w_{p_\ell}(t_\ell)]=\E[N_{p_1}(t_1)N_{p_2}(t_2) \cdots N_{p_\ell}(t_\ell)].$$
Now using trace formula (\ref{trace formula C_n}), we have  
\begin{align*}
w_{p_k}(t_{k}) & = \frac{1}{\sqrt{n}} \Big(\Tr(C_n(t_k))^{p_k} - \E[\Tr(C_n(t_k))^{p_k}]\Big)\\
 &= \frac{1}{n^{\frac{p_k - 1}{2}}} \sum_{A_{p_k}} \Big( b_{j_{k,1}}(t_k)\cdots b_{j_{k,p_k}}(t_k) - \E[b_{j_{k,1}}(t_k)\cdots b_{j_{k,p_k}}(t_k)]\Big).
\end{align*}
Note that in the above summation $(j_{k,1},j_{k,2},\ldots,j_{k,p_k})\in A_{p_k}$.
Therefore 
\begin{align}\label{eqn:expectation_thm2}
&\quad \E[w_{p_1}(t_1) \cdots w_{p_\ell}(t_\ell)] \\
&= \frac{1}{n^{\frac{p_1 + p_2 + \cdots +p_\ell -\ell}{2}}} \sum_{A_{p_1}, A_{p_2}, \ldots, A_{p_\ell}} \E\big[ (b_{J_1} - \E b_{J_1}) (b_{J_2} - \E b_{J_2})  \cdots (b_{J_\ell} - \E b_{J_\ell})\big].\nonumber
 \end{align}


Now for a fixed $J_1,J_2,\ldots,J_\ell$, if there exists a $k\in\{1,2,\ldots,\ell\}$ such that $J_k$ is not connected with any $J_i$ for $i\neq k$, then 
$$\E\big[ (b_{J_1} - \E b_{J_1}) (b_{J_2} - \E b_{J_2})  \cdots (b_{J_\ell} - \E b_{J_\ell})\big]=0$$
due to the independence of Brownian motions $\{b_i\}_{i\geq 0}$.

Therefore $J_1,J_2,\ldots,J_\ell$ must form  clusters with each cluster length greater equal to two, that is, each cluster should contain at least two vectors. Suppose $C_1,C_2,\ldots,C_s$ are the clusters formed by  vectors $J_1,J_2,\ldots,J_\ell$ and   $|C_i|\geq 2$ for all $1\leq i \leq s$ where $|C_i|$ denotes the length of the cluster $C_i$. Observe that $\sum_{i=1}^s |C_i|=\ell$.   

If there exists  a cluster $C_j$ among $C_1,C_2,\ldots,C_s$ such that $|C_j|\geq 3$, then from Theorem \ref{thm:circovar} and Lemma \ref{lem:maincluster}, we have  
\begin{align*}
\frac{1}{n^{\frac{p_1 + p_2 + \cdots +p_\ell -\ell}{2}}} \sum_{A_{p_1}, A_{p_2}, \ldots, A_{p_\ell}} \E\big[ (b_{J_1} - \E b_{J_1}) (b_{J_2} - \E b_{J_2})  \cdots (b_{J_\ell} - \E b_{J_\ell})\big]=o(1).
\end{align*}  
Therefore, if $\ell$ is odd then there will be a cluster of odd length and hence 
$$ \lim_{n\tends \infty}  \E[w_{p_1}(t_1) \cdots w_{p_\ell}(t_\ell)] = 0.$$ 
 
Similarly, if $\ell$ is even then the contribution due to $\{ J_1, J_2, \ldots, J_\ell \}$ to  $ \E[w_{p_1}(t_1) \cdots w_{p_\ell}(t_\ell)]$ is $O(1)$ only when $\{ J_1, J_2, \ldots, J_\ell\}$  decomposes into clusters of length 2. Therefore from \eqref{eqn:expectation_thm2}, we get
\begin{align*} \label{eq:multisplit}
 & \quad \lim_{n\tends \infty}  \E[w_{p_1}(t_1) \cdots w_{p_\ell}(t_\ell)]\\
 & =\lim_{n\to\infty} \frac{1}{n^{\frac{p_1 + p_2 + \cdots +p_\ell -\ell}{2}}} \sum_{A_{p_1}, A_{p_2}, \ldots, A_{p_\ell}} \E\big[ (b_{J_1} - \E b_{J_1}) (b_{J_2} - \E b_{J_2})  \cdots (b_{J_\ell} - \E b_{J_\ell})\big]\\
 & =\lim_{n\to\infty} \frac{1}{n^{\frac{p_1 + p_2 + \cdots +p_\ell -\ell}{2}}} \sum_{\pi \in \mathcal P_2(\ell)} \prod_{i=1}^{\frac{\ell}{2}}  \sum_{A_{p_{y(i)}},\ A_{p_{z(i)}}} \E\big[ (b_{J_{y(i)}} - \E b_{J_{y(i)}}) (b_{J_{z(i)}} - \E b_{J_{z(i)}})\big],
 \end{align*}
where  $\pi = \big\{ \{y(1), z(1) \}, \ldots , \{y(\frac{\ell}{2}), z(\frac{\ell}{2})  \} \big\}\in \mathcal P_2(\ell)$ and $\mathcal P_2(\ell)$ is the set of all pair partition of $ \{1, 2, \ldots, \ell\} $. Using Theorem \ref{thm:circovar}, from the last equation we get
\begin{equation}\label{eqn:product of expectation}
 \lim_{n\tends \infty}  \E[w_{p_1}(t_1) \cdots w_{p_\ell}(t_\ell)]
  =\sum_{\pi \in P_2(\ell)} \prod_{i=1}^{\frac{\ell}{2}} \lim_{n\tends \infty} \E[w_{p_{y(i)}} (t_{y(i)}) w_{p_{z(i)}} (t_{z(i)})].
  \end{equation}
Now from Corollary \ref{cor:existence of Gaussian process}, there exists a centred Gaussian process $\{N_p(t);t\geq 0\}_{p\geq 2}$ such that 
$$\E(N_p(t_1)N_q(t_2))=\lim_{n\to\infty}\E(w_p(t_1)w_q(t_2)).$$
Therefore using Wick's formula,  from \eqref{eqn:product of expectation} we get 
\begin{align*}
\lim_{n\tends \infty}  \E[w_{p_1}(t_1) \cdots w_{p_\ell}(t_\ell)]
  &=\sum_{\pi \in \mathcal P_2(\ell)} \prod_{i=1}^{\frac{\ell}{2}} \lim_{n\tends \infty} \E[w_{p_{y(i)}} (t_{y(i)}) w_{p_{z(i)}} (t_{z(i)})]\\
 & =\sum_{\pi \in \mathcal P_2(\ell)} \prod_{i=1}^{\frac{\ell}{2}} \E[N_{p_{y(i)}} (t_{y(i)}) N_{p_{z(i)}}  (t_{z(i)})] \\
 &=\E[ N_{p_1}(t_1)N_{p_2}(t_2) \cdots N_{p_\ell}(t_\ell) ].
\end{align*}
This completes the proof of Theorem \ref{thm:cirmulti}.

\end{proof}

\section{Proof of Theorem \ref{thm:process}}\label{sec:process convergence}
 We use Theorem \ref{thm:cirmulti} and some \textit{combinatorial techniques} to prove Theorem \ref{thm:process}. For fixed $p \geq 2$, we want to show that $ \{ w_p(t) ; t \geq 0\} \stackrel{\mathcal D}{\rightarrow} \{N_p(t) ;t \geq 0\} $, where the existence of $\{N_p(t) ; t \geq 0\}$ is coming from Corollary \ref{cor:existence of Gaussian process}. We first state some results which will be used in the proof of Theorem \ref{thm:process}.

Suppose $C_{\infty}:= C[0, \infty)$ be the space of all continuous real-valued function on $[0, \infty)$. Then $(C_{\infty}, \rho )$ is a metric space with metric
$$ \rho(X,Y) = \sum_{k=1}^{\infty} \frac{1}{2^k} \frac{\rho_k(X,Y)}{1+ \rho_k(X,Y)} \ \ \forall \ X, Y \in C_{\infty},$$  
where $\rho_k(X,Y) = \sup_{0 \leq t \leq k} |X(t)- Y(t)|$ is the usual metric on $C_k := C[0,k]$, the space of all continuous real-valued function on $[0, k]$. Suppose $\mathcal{C_{\infty}}$ and $\mathcal{C}_k$ be the $\sig$-field generated by the open sets of $C_{\infty}$ and $C_k$, respectively. Let $\{ \P_n \}$ and $\P$ be probability measure on $( C_{\infty},\mathcal{C_{\infty}} )$. If 
$$ \P_n f := \lim_{n\tends \infty} \int_{C_{\infty}} f d \P_n \tends \P f := \int_{C_{\infty}} f d \P ,$$
 for every bounded, continuous real-valued function $f$ on $C_{\infty}$, then we say $\P_n$ converge to $\P$ \textit{weakly} or in \textit{distribution}. we denote it by $\P_n \stackrel{\mathcal D}{\rightarrow} \P $. The following result gives if and only if condition for the weak convergence of $\P_n$ to $\P$.
%
	
\begin{result} \textbf{(Theorem 3, \cite{Ward1970})}  \label{result:process convergance}
	Suppose $\{ \P_n \}$ and $\P$ are probability measures on  $( C_{\infty},\mathcal{C_{\infty}} )$. Then  $\P_n \stackrel{\mathcal D}{\rightarrow} \P$ if and only if:
	
	\noindent \textbf{(i)} the finite-dimensional distributions of $\P_n$ converge weakly to those of $\P$, that is,
	\begin{equation} \label{eqn:P_nGamma convergence}
	\P_n \pi^{-1}_{t_1\cdots t_r} \stackrel{\mathcal D}{\rightarrow} \P \pi^{-1}_{t_1\cdots t_r} \ \ \forall \ r \in \mathbb{N} \ \mbox{and} \ \forall \ r\mbox{-tuples} \ t_1, \ldots, t_r, \nonumber
	\end{equation}
	where $\pi_{t_1\cdots t_r}$ is the natural projection from $C_{\infty}$ to $\mathbb{R}^r$ for all $r \in \mathbb{N}$, and 
	\vskip5pt
	
	\noindent \textbf{(ii)} the sequence $\{ \P_n \} $ is tight.
\end{result}
 Suppose $\{ \boldsymbol{X_n} \} = \{ X_n (t); t \geq 0\}$ is a sequence of continuous stochastic process defined on some probability space $( \Omega, \mathcal{F}, P )$, that is, $\boldsymbol{X_n}$ is a measurable map from  $( \Omega, \mathcal{F}, P)$ to $(C_{\infty},\mathcal{C_{\infty}})$.
Let $\{ \P_n\}$ be the sequence of probability measure on  $(C_{\infty},\mathcal{C_{\infty}})$ induced by $ \boldsymbol{X_n}$.
The following result provide a sufficient condition for the tightness of the probability measure $\{ \P_n\}$.
\begin{result} \textbf{(Theorem I.4.3, \cite{Ikeda1981})}\label{result:tight2}
	Suppose $\{\boldsymbol{X_n}\}$ = $\{ X_n(t); t \geq 0 \}$, $n\in \mathbb{N}$ be a sequence of continuous processes satisfying the following two conditions:
	
    \noindent	\textbf{(i)} there exists positive constants $M$ and $\gamma$ such that
	$$\E{|X_n(0)|^ \gamma } \leq M \ \ \ \forall \ n \in \mathbb{N},$$ 	 
    \noindent	\textbf{(ii)} there exists positive constants $\alpha, \ \beta$ and $M_T$, $T = 1, 2, \ldots,$ such that
	$$\E{|X_n(t)- X_n(s)|^ \alpha } \leq M_T |t-s|^ {1+ \beta} \ \ \ \forall \ n \in \mathbb{N} \ \mbox{and } t,s \in[0,\ T],(T= 1, 2, \ldots,).$$
	Then the sequence $\{ \boldsymbol{X_n} \} $ is tight and the probability measure $\{ \P_n \}$, induced by $\{ \boldsymbol{X_n} \} $ is also tight.	
\end{result}
In our setup $X_n(t)$ is $w_p(t)$, where $w_p(t)$ is as defined in (\ref{eqn:w_p(t)}). Now from Result \ref{result:process convergance} and Result \ref{result:tight2}, it is clear that, to prove Theorem \ref{thm:process}, it is sufficient to prove the following two propositions: 

\begin{proposition} \label{pro:finite convergence}
	For each $ p\geq  2$, suppose $0<t_1<t_2 \cdots <t_r$. Then as $n \tends \infty$
	\begin{equation}
	(w_{p}(t_1), w_{p}(t_2), \ldots , w_{p}(t_r)) \stackrel{\mathcal D}{\rightarrow} (N_{p}(t_1), N_{p}(t_2), \ldots , N_{p}(t_r)).
	\end{equation}	
\end{proposition}

\begin{proposition} \label{pro:tight}
	For each $p \geq2$, there exists positive constants $M$ and $\gamma$ such that
	\begin{equation} \label{eq:tight3}
	\E{|w_p(0)|^ \gamma } \leq M \ \ \ \forall \ n \in \mathbb{N}, 	 
	\end{equation}
	there exists positive constants $\alpha, \ \beta $ and $M_T$, $T= 1, 2, \ldots,$ such that
	\begin{equation} \label{eq:tight4}
	\E{|w_p(t)- w_p(s)|^ \alpha } \leq M_T |t-s|^ {1+ \beta} \ \ \ \forall \ n \in \mathbb{N} \ \mbox{and } t,s \in[0,\ T],(T= 1, 2, \ldots,).
	\end{equation}
	Then $\{ w_p(t) ; t \geq 0 \}$ is tight.
\end{proposition} 
    Proposition \ref{pro:finite convergence} describes finite dimensional joint fluctuation of $\{w_p(t);t \geq 0\}$ as $n \to \infty$ and Proposition \ref{pro:tight} describes tightness of the process $\{w_p(t);t \geq 0\}$.

\begin{proof}[Proof of Proposition \ref{pro:finite convergence}]
	Proposition \ref{pro:finite convergence} is a particular case of Theorem \ref{thm:cirmulti}. If we take $p_i = p$, in Theorem \ref{thm:cirmulti}, then we get 
	\begin{align*}
	\lim_{n\to \infty} & \P  \Big( w_{p}(t_1)  \leq a_1,       w_{p}(t_2)  \leq a_2,\ldots, w_{p}(t_r)  \leq a_r \Big) \nonumber \\
	&= \P  \Big( N_{p}(t_1)  \leq a_1,                              N_{p}(t_2)  \leq a_2, \ldots, N_{p}(t_r)  \leq a_r \Big).
	\end{align*}
	This shows that, as $n \tends \infty$
	$$(w_{p}(t_1), w_{p}(t_2), \ldots , w_{p}(t_r)) \stackrel{\mathcal D}{\rightarrow} (N_{p}(t_1), N_{p}(t_2), \ldots , N_{p}(t_r)).$$
	This complete the proof of Proposition \ref{pro:finite convergence}.
\end{proof}

\begin{proof}[Proof of Proposition \ref{pro:tight}] Since $w_p(0)= 0$ for all $p \geq 2$. Therefore $\E{|w_p(0)|^ \gamma } = 0$ for all $\gamma$. Hence for any $M >0$,
	$$\E{|w_p(0)|^ \gamma } \leq M \ \  \forall \ n \in \mathbb{N}.$$
	This shows (\ref{eq:tight3}) is true.
	
We shall prove (\ref{eq:tight4}) of Proposition \ref{pro:tight} for $\alpha = 4$ and $\beta = 1$. First recall $w_p(t)$ from (\ref{eqn:w_p(t)}),
$$ w_p(t) = \frac{1}{\sqrt{n}} \bigl\{ \Tr(C_n(t))^p - \E[\Tr(C_n(t))^p]\bigr\}.$$
Suppose $p \geq 2$ is fixed and $t,s \in[0,T] $, for some $T \in \mathbb{N}$. Then
\begin{align*}
 w_p(t) - w_p(s) &= \frac{1}{\sqrt{n}} \big[ \Tr(C_n(t))^p - \Tr(C_n(s))^p - \E[\Tr(C_n(t))^p - \Tr(C_n(s))^p]\big] \\
  \E[w_p(t) - w_p(s)]^4 &= \frac{1}{n^2} \E \big[ \Tr(C_n(t))^p - \Tr(C_n(s))^p - \E[\Tr(C_n(t))^p - \Tr(C_n(s))^p]\big]^4.
\end{align*}
Since we know that for any $r \in \mathbb{N}$ 
\begin{equation} \label{eqn:nth power}
|x_1 + x_2 + \cdots + x_n|^r \leq 2^{r-1} (|x_1|^r + |x_2|^r + \cdots + |x_n|^r).  
\end{equation}
Therefore from (\ref{eqn:nth power})
\begin{align} \label{eqn:|R_1|+|R_2|}
\E[w_p(t) - w_p(s)]^4 & \leq \frac{2^3}{n^2} \E \big[ \Tr(C_n(t))^p - \Tr(C_n(s))^p\big]^4 \nonumber \\
   &  \ \ \ \ \ + \frac{2^3}{n^2} \E \big[ \E [\Tr(C_n(t))^p - \Tr(C_n(s))^p]\big]^4 \nonumber \\
& \leq \frac{2^3}{n^2} \E \big[ \Tr(C_n(t))^p - \Tr(C_n(s))^p\big]^4 \nonumber \\
   &   \ \ \ \ \ + \frac{2^3}{n^2} \big[ \E [\Tr(C_n(t))^p - \Tr(C_n(s))^p]\big]^4 \nonumber \\
& = R_1 + R_2, \  \mbox{say}. 
\end{align}
We shall prove (\ref{eq:tight4}) of Proposition \ref{pro:tight} in two Steps. In Step 1, we shall show
\begin{equation} \label{eqn:|R_1|}
|R_1| \leq M^T_7 (t-s)^2 \ \ \ \forall \ n \in \mathbb{N} \ \mbox{and } t,s \in[0,\ T],
\end{equation}
 where $M^T_7$ is a positive constant, which depends only on $p$ and $T$.
In Step 2, we shall show 
\begin{equation} \label{eqn:|R_2|}
|R_2| \leq M^T_{10} (t-s)^2 \ \ \ \forall \ n \in \mathbb{N} \ \mbox{and } t,s \in[0,\ T],
\end{equation}
 where $M^T_{10}$ is a positive constant, which depends only on $p$ and $T$. 
 \vskip8pt
\noindent \textbf{Step 1: Proof of (\ref{eqn:|R_1|}).}
\vskip8pt
 Since $0< s <t $, we have
\begin{align*}
C_n(t) &= C_n(t) - C_n(s) + C_n(s)
 = C'_n(t-s) + C_n(s)
\end{align*}
where $C'_n(t-s)$ is a circulant matrix with entries $\{ b_n(t) -b_n(s)\}_{n \geq 0}$. As $b_n(t)$ is the standard Brownian motion, $b_n(t) -b_n(s)$ has same distribution as $b_n(t-s)$ and hence 
\begin{equation} \label{eqn:E[C'_n]}
\E[\Tr (C_n(t-s))^p] = \E[\Tr (C'_n(t-s))^p].
\end{equation}
Now
 \begin{equation*}
 (C_n(t))^p  = [C'_n(t-s)+ C_n(s)]^p.
 \end{equation*} 
 Use binomial expansion on right hand side of the above equation, we get 
\begin{align} \label{eqn:C_n^p(t)- C_n^p(s)}
(C_n(t))^p - (C_n(s))^p & = (C'_n(t-s))^p + \sum_{d=1}^{p-1} \binom{p}{d} (C'_n(t-s))^d (C_n(s))^{p-d}. 
\end{align} 
Therefore
\begin{align}
[\Tr (C_n(t))^p- \Tr (C_n(s))^p]^4 & = \big[ \Tr[ (C'_n(t-s) )^p] + \sum_{d=1}^{p-1} \binom{p}{d} \Tr [(C'_n(t-s))^d (C_n(s))^{p-d}]  \big]^4 \nonumber \\
 & \leq  2^3 [\Tr (C'_n (t-s))^p]^4 \nonumber \\
  &  \ \ \ \ \ + 2^3 \big[  \sum_{d=1}^{p-1}\binom{p}{d} \Tr[ (C'_n(t-s))^d (C_n(s))^{p-d}] \big]^4. \nonumber 
 \end{align}
 Note that $\binom{p}{d} \leq 2^p$, for all $1 \leq d \leq p-1$. So, we have
 \begin{align}
 [\Tr (C_n(t))^p- \Tr (C_n(s))^p]^4 & \leq  2^3 [\Tr (C'_n (t-s))^p]^4 \nonumber \\
 &  \ \ \ \ \ + 2^3 \big[  \sum_{d=1}^{p-1} 2^p \Tr[ (C'_n(t-s))^d (C_n(s))^{p-d}] \big]^4 \nonumber \\
 & \leq  2^3 [\Tr (C'_n(t-s))^p]^4 \nonumber \\
  &  \ \ \ \ \ +  2^{4p+6} \sum_{d=1}^{p-1}  [\Tr (C'_n(t-s))^d (C_n(s))^{p-d}]^4. \nonumber  
\end{align}
Therefore from the above inequality, we get
\begin{align*}
R_1 & = \frac{2^3}{n^2} \E \big[ \Tr(C_n(t))^p - \Tr(C_n(s))^p\big]^4 \\
& \leq  \frac{2^6}{n^2} \E [\Tr (C'_n(t-s))^p]^4 +\frac{2^{4p+9}}{n^2} \E \sum_{d=1}^{p-1} [\Tr (C'_n(t-s))^d (C_n(s))^{p-d}]^4.  
\end{align*}
From (\ref{eqn:E[C'_n]}) we have, $\E[\Tr (C_n(t-s))^p] = \E[\Tr (C'_n(t-s))^p]$, therefore
\begin{align} \label{eqn:R_1}
 R_1 & \leq  \frac{2^6}{n^2} \E [\Tr (C_n(t-s))^p]^4 + \sum_{d=1}^{p-1} \frac{ 2^{4p+9} }{n^2} \E  [\Tr (C_n(t-s))^d (C_n(s))^{p-d}]^4 \nonumber \\
 & =  T_0 + \sum_{d=1}^{p-1} T_d, \ \mbox{say}.
\end{align}
We shall show that
\begin{align}
T_0 & \leq M_3(t-s)^{2p} \ \ \mbox{and} \label{eqn:T_0} \\
T_d & \leq M^{d}_6 (t-s)^{2d} \ \   \mbox{for} \ 1 \leq d \leq (p-1), \label{eqn:T_d}
\end{align}
where $M_3$ and $M^{d}_6$ are positive constants. Now we define subsets $D_{p_4}$ and $D^c_{p_4}$ of $ A_{p} \times A_{p} \times A_{p} \times A_{p}$, which will be used in the proof of (\ref{eqn:T_0}) and (\ref{eqn:T_d}).
\begin{definition} \label{def:D_{p_4}}
	For $p \geq 2$, define $D_{p_4}$ is a subset of  $ A_{p} \times A_{p} \times A_{p} \times A_{p}$ such that  $ (I_p, J_p, K_p, L_p) \in D_{p_4}$ if 
	\begin{enumerate} 
		\item[(i)] $I_p$ completely matches  with one of $J_p, K_p, L_p$ and the remaining two of $J_p, K_p, L_p$ completely matches with themselves,
		\item[(ii)] suppose $I_p$ completely matches  with  $J_p$ and $K_p$ completely matches  with  $L_p$, then all entries of $S_{I_p} \cup S_{K_p}$ are distinct. 
	\end{enumerate}
\end{definition}
\begin{definition} \label{def:D^c_{p_4}}
	For $p \geq 2$, define $D^c_{p_4}$ = $ A_{p} \times A_{p} \times A_{p} \times A_{p} \ \backslash \ D_{p_4}$. That is, if $ (I_p, J_p, K_p, L_p) \in D^c_{p_4}$, then there is at least one self-matching in one of  $I_p,\ J_p,\ K_p$ and $L_p $.
	
\end{definition}
\noindent \textbf{Proof of (\ref{eqn:T_0}):}
\vskip8pt
We first calculate the term $T_0$ of (\ref{eqn:R_1}).
\begin{align} \label{eqn:T_01}
T_0 = & \frac{2^6}{n^2} \E [\Tr (C_n(t-s))^p]^4 \nonumber \\
 = & \frac{2^6}{n^2} \E [\frac{n}{ n^\frac{p}{2}} \sum_{ A_p} b_{i_1}(t-s) b_{i_2} (t-s) \cdots b_{i_p} (t-s) ]^4. 
\end{align}

For $I_p = (i_1, i_2, \ldots, i_p) \in A_p$, we define $b_{I_p}(t-s) = b_{i_1}(t-s) b_{i_2} (t-s) \cdots b_{i_p} (t-s)$.
Similarly, for $J_p$, $K_p$ and $L_{p} \in A_p$ we define $  b_{J_p} (t-s), b_{K_p} (t-s)$ and $b_{L_p} (t-s)$, respectively. Now (\ref{eqn:T_01}) can be written as

\begin{align} \label{eqn:S_0+S^c_0}
T_0= & \frac{2^6}{n^{2p-2}} \E [ \sum_{A_p, A_p, A_p, A_p} b_{I_p}(t-s) b_{J_p} (t-s)b_{K_p} (t-s) b_{L_p} (t-s) ] 
 =  S_0 + S^c_0,
\end{align}
where
$$ S_0 = \frac{2^6}{n^{2p-2}} \E [ \sum_{D_{p_4}} b_{I_p}(t-s) b_{J_p} (t-s)b_{K_p} (t-s) b_{L_p} (t-s) ],$$
$$ S^c_0 = \frac{2^6}{n^{2p-2}} \E [ \sum_{D^c_{p_4}} b_{I_p}(t-s) b_{J_p} (t-s)b_{K_p} (t-s) b_{L_p} (t-s) ]$$
and $D_{p_4}$, $D^c_{p_4}$ are as defined in (\ref{def:D_{p_4}}), (\ref{def:D^c_{p_4}}), respectively. We shall see in $T_0$, the maximum contribution will come from $S_0$. More precisely $S_0$ will be $O(1)$ and $S^c_0$ will be $o(1)$, as $n \tends \infty$. 
%

Now, we calculate the first term $S_0$ of (\ref{eqn:S_0+S^c_0}), for $(I_p, J_p, K_p, L_p) \in {D_{p_4}}$. Suppose $I_p$ completely matches with $J_p$, say and $K_p$ completely matches with $L_p$, and each entries of $I_p$ and $K_p$ are distinct. Hence $S_0$ will be

$$ S_0 = \frac{2^6}{n^{2p-2}} 3 (p!)^2 \E [ \sum_{D_{p_2}} (b_{I_p}(t-s))^2 (b_{K_p} (t-s))^2 ],$$ 
where $D_{p_2}= \{ (I_p, K_p) \in A_p \times A_p : \mbox{all entries of } S_{I_p} \cup S_{K_p} \mbox{ are distinct} \}$. The factor $(p!)^2$ is coming because $I_p$ can match with the given vector $J_p$ in $p!$ ways and $K_p$ can match with the given vector $L_p$ in $p!$ ways. The factor $3$ is coming because there are three ways to partition a set with 4 elements into two subsets where each subset contains two elements. Since all entries of $D_{p_2}$ are distinct, therefore
\begin{align*}
 S_0 &= \frac{2^63}{n^{2p-2}} (p!)^2 \sum_{D_{p_2}} \E[(b_{I_p}(t-s))^2 (b_{K_p} (t-s))^2] \\
 & = \frac{2^63}{n^{2p-2}} (p!)^2 \sum_{D_{p_2}} (t-s)^{2p} \\
 & \leq \frac{2^63}{n^{2p-2}} (p!)^2 \sum_{A_p, A_p} (t-s)^{2p}.
\end{align*}
From (\ref{def:A_p}), observe that $A_p = \cup_{x=1}^{p-1} A_{p,x}$. Now from the last expression, we get
\begin{align*}
S_0 & \leq 2^63 (t-s)^{2p} (p!)^2 \sum_{x,y=0}^{p-1} \frac{|A_{p,x}|}{n^{p-1}}\frac{|A_{p,y}|}{n^{p-1}}.
\end{align*}
For each fixed $p \geq 2$, from Result \ref{ft:variance}, we get
$$ \lim_{n\tends \infty} \frac{|A_{p,x}|}{n^{p-1}} = f_p(x) \ \ \ \mbox{and} \ \ \lim_{n\tends \infty} \frac{|A_{p,y}|}{n^{p-1}} = f_p(y),$$
where $f_p(x)=\frac{1}{(p-1)!}\sum_{k=0}^{x}(-1)^k\binom{p}{k}(x-k)^{p-1}.$
Therefore there exists a positive constant $M_1$, which depends only on $p$, such that 
$$ 2^63 (p!)^2 | \sum_{x,y=0}^{p-1} \frac{|A_{p,x}|}{n^{p-1}}\frac{|A_{p,y}|}{n^{p-1}}| \leq M_1 \ \ \forall \ n \in \mathbb{N}.$$
Hence 
\begin{equation} \label{eqn:S_0}
|S_0| \leq M_1 (t-s)^{2p}  \ \ \forall \ n \in \mathbb{N}.
\end{equation}
Now, we calculate the second term $S^c_0$ of (\ref{eqn:S_0+S^c_0}).
First observe that $S^c_0= 0$ when an odd power of $b_i(t)$ appears in $S^c_0$ because $b_i(t)$ is the standard Brownian motion for all $i \geq 0$.
 From the calculation of $S_0$, it is clear that $S_0^c= o(1)$ as $(I_p, J_p, K_p, L_p) \in D^c_{p_4} $.

 Let us look at a special case for $(I_p, J_p, K_p, L_p) \in D^c_{p_4} $. Suppose there is a self-matching in $I_p$ and $J_p$, and $K_p$ matches completely with $L_p$. Now suppose $x$ many entries from $I_p$ matches exactly with $x$ many entries from $J_p$, where $x \leq p$. Then a typical term in $S^c_0$ will be of the following form
\begin{align*} 
\E[b^2_{i_1}(t-s) b^2_{i_2} (t-s) \cdots b^2_{i_x} (t-s) b_{i_{x+1}}(t-s) \cdots b_{i_p} (t-s) b_{j_{x+1}}(t-s) \cdots b_{j_p}  (t-s) \\ b^2_{k_1}(t-s) b^2_{k_2} (t-s) \cdots b^2_{k_p} (t-s) ].
\end{align*}
Note that the last expression will be non-zero if, each random variable appears at least twice. Also there are three constraints among the indices, namely; $i_1 + i_2 + \cdots+ i_p = 0 (\mbox{mod $n$}$), $j_1 + j_2 + \cdots+ j_p = 0 (\mbox{mod $n$}$) and $k_1 + k_2 + \cdots+ k_p = 0 (\mbox{mod $n$}$). Thus we have at most $[ x + (\frac{p-x}{2}-1) + (\frac{p-x}{2}-1) + (p-1) ] = 2p-3 $ free choice in $ D^c_{p_4} $. Similarly in the other cases of self-matching with $(I_p, J_p, K_p, L_p) \in D^c_{p_4} $, the number of free vertices will be bounded by $(2p-3)$. Hence the maximum contribution due to $ D^c_{p_4} $ will be $ O(n^{2p-3}).$ Also $2k$-th moment of $b_i(t)$ is $(t^k \frac{(2k)!}{2^k k!})$, therefore
\begin{align} \label{eqn:S^c_01}
S^c_0 & = \frac{2^6}{n^{2p-2}}  \E [  \sum_{D^c_{p_4}} b_{I_p}(t-s) b_{J_p} (t-s)b_{K_p} (t-s) b_{L_p} (t-s) ]  \nonumber \\
    & = 2^6 (t-s)^{2p} g(p) O\big( \frac{n^{2p-3}}{n^{2p-2}} \big)\nonumber \\
    & = 2^6 (t-s)^{2p} g(p) O (n^{-1} ) \nonumber \\
    & = 2^6 (t-s)^{2p} g(p) h(n),
\end{align}
where $g(p)$ is some function of $p$ and $h(n) = O(n^{-1})$, that is, $h(n) \tends 0$ as $ n \tends 0$. Therefore there exists a positive constant $M_2$, which depends only on $p$ such that 
$$2^6 | g(p) h(n)| \leq M_2  \ \ \forall \ n \in \mathbb{N}.$$
Hence 
\begin{equation} \label{eqn:S^c_0}
S^c_0 \leq M_2 (t-s)^{2p}  \ \ \forall \ n \in \mathbb{N}.
\end{equation}
Now, using (\ref{eqn:S_0}) and (\ref{eqn:S^c_0}) in (\ref{eqn:S_0+S^c_0}), we get 
\begin{align*} 
T_0 & \leq M_3 (t-s)^{2p}  \ \ \forall \ n \in \mathbb{N},
\end{align*}
where $M_3 = M_1 + M_2$.
\vskip5pt
\noindent \textbf{Proof of $ (\ref{eqn:T_d})$:}
\vskip8pt
We calculate the term $T_d$ of (\ref{eqn:R_1}) for $1 \leq d \leq (p-1).$
 First recall 
\begin{align*}
T_d & = \frac{2^{4p+9}}{n^2} \E  [\Tr (C_n(t-s))^d (C_n(s))^{p-d}]^4 \\
 & = \frac{2^{4p+9} }{n^2} \E [\frac{n}{ n^\frac{p}{2}} \sum_{ A_p} b_{i_1}(t-s) b_{i_2} (t-s) \cdots b_{i_d} (t-s) b_{i_{d+1}}(s) b_{i_{d+2}}(s) \cdots b_{i_p}(s) ]^4. \nonumber
\end{align*}
For $I_p= (i_1, i_2, \ldots, i_d, i_{d+1} ,\ldots , i_p) \in A_p$ and $1 \leq d \leq p-1$, we define $I_d=(i_1, i_2, \ldots, i_d)$ and $I_{p-d} = ( i_{d+1} , i_{d+2}, \ldots , i_p)$. Similarly, for $J_p, K_p$ and $L_p \in A_p$ we define $J_d, J_{p-d}; K_d, K_{p-d}$ and $L_d, L_{p-d}$, respectively. 

Also define
$$
b_{I_d}(t-s) = b_{i_1}(t-s) b_{i_2} (t-s) \cdots b_{i_d} (t-s), \ b_{I_{p-d}}(s)= b_{i_{d+1}}(s) b_{i_{d+2}} (s) \cdots b_{i_{p}} (s).
$$
Similarly, we define $  b_{J_d} (t-s), b_{J_{p-d}} (s); b_{K_d} (t-s), b_{K_{p-d}} (s)$ and  $b_{L_d} (t-s), b_{L_{p-d}} (s)$, respectively. Using these notation, $T_d$ can be written as


\begin{align*}
T_d & = \frac{2^{4p+9} }{n^{2p-2}} \E [ \sum_{A_p, A_p, A_p, A_p} b_{I_d}(t-s)b_{I_{p-d}}(s) b_{J_d} (t-s) b_{J_{p-d}} (s)\\
 & \ \ \ \ \ \ b_{K_d} (t-s)b_{K_{p-d}} (s) b_{L_d} (t-s)b_{L_{p-d}} (s) ]. \nonumber 
\end{align*}
Now by the independent increment property of the Brownian motion, we have
\begin{align}
T_d = & \frac{2^{4p+9} }{n^{2p-2}}  \sum_{A_p, A_p, A_p, A_p}  \E[b_{I_d}(t-s)  b_{J_d} (t-s) b_{K_d} (t-s) b_{L_d} (t-s) ] \nonumber\\
& \ \ \E[ b_{I_{p-d}}(s) b_{J_{p-d}} (s) b_{K_{p-d}} (s) b_{L_{p-d}} (s) ] \nonumber \\
 = & S_d + S^c_d, \label{eqn:S_d+S^c_d} 
\end{align}
where
\begin{align*}
S_d & =  \frac{2^{4p+9} }{n^{2p-2}}   \sum_{D_{p_4}}  \E[b_{I_d}(t-s)  b_{J_d} (t-s) b_{K_d} (t-s) b_{L_d} (t-s) ] \nonumber \\
& \ \ \ \ \ \E[ b_{I_{p-d}}(s) b_{J_{p-d}} (s) b_{K_{p-d}} (s) b_{L_{p-d}} (s) ],
\end{align*}

\begin{align*}
S^c_d & =  \frac{2^{4p+9} }{n^{2p-2}} \sum_{D^c_{p_4}}  \E[b_{I_d}(t-s)  b_{J_d} (t-s) b_{K_d} (t-s) b_{L_d} (t-s) ] \nonumber \\ 
& \ \ \ \ \ \E[ b_{I_{p-d}}(s) b_{J_{p-d}} (s) b_{K_{p-d}} (s) b_{L_{p-d}} (s) ]
\end{align*}
and $D_{p_4}$, $D^c_{p_4}$ are as defined in (\ref{def:D_{p_4}}), (\ref{def:D^c_{p_4}}), respectively.

Now, we calculate $S_d$. For maximum contribution, the total number of free variables in
$$ \E[b_{I_d}(t-s)  b_{J_d} (t-s) b_{K_d} (t-s) b_{L_d} (t-s) ] \E[ b_{I_{p-d}}(s) b_{J_{p-d}} (s) b_{K_{p-d}} (s) b_{L_{p-d}} (s) ],$$
is $[ d+d + (p-d) -1 + (p-d)-1]= 2p-2$, which is same as the order of $n$ in the denominator of $T_d$. Therefore by similar calculation as we have done for $S_0$, we get
\begin{equation*}
 S_d \leq 2^{4p+9} (t-s)^{2d}s^{2(p-d)} (3)^2 (d!(p-d)!)^2 \sum_{x,y=0}^{p-1} \frac{|A_{p,x}|}{n^{p-1}}\frac{|A_{p,y}|}{n^{p-1}}.
\end{equation*}
Since $s\in [0,\ T]$ and $ \lim_{n\tends \infty} \frac{|A_{p,x}|}{n^{p-1}} = f_p(x).$ Therefore there exists a positive constant $M^d_4$, which depends only  on $p$, $d$ and $T$, such that 
$$  2^{4p+9} 3^2 s^{2(p-d)} (d!(p-d)!)^2 \arrowvert  \sum_{x,y=0}^{p-1} \frac{|A_{p,x}|}{n^{p-1}}\frac{|A_{p,y}|}{n^{p-1}} \arrowvert  \leq M^d_4  \ \ \forall \ n \in \mathbb{N}.$$
Hence 
\begin{equation} \label{eqn:S_d}
S_d \leq M^d_4 (t-s)^{2d}  \ \ \forall \ n \in \mathbb{N}.
\end{equation}
Now following the similar arguments used in the proof of (\ref{eqn:S^c_0}) and (\ref{eqn:S_d}), we get 
\begin{equation*}
S^c_d \leq M^d_5 (t-s)^{2d}  \ \ \forall \ n \in \mathbb{N},
\end{equation*}
where $M^d_5$ is a positive constant, which depends only  on $p$, $d$ and $T$. Hence

\begin{align*} 
T_d &\leq  M^d_6 (t-s)^{2d} \ \ \forall \ n \in \mathbb{N},
\end{align*}
where $M^d_6 = M^d_4 + M^d_5$.

Use (\ref{eqn:T_0}) and (\ref{eqn:T_d}) in (\ref{eqn:R_1}), we get 
\begin{align*}
R_1 &\leq  M_3 (t-s)^{2p} + \sum_{d=1}^{p-1} M^d_6 (t-s)^{2d} \ \ \ \forall \ n \in \mathbb{N} \ \mbox{and } t,s \in[0,\ T].
\end{align*}
As $s,t \in [0, \ T]$, we get
\begin{equation*}
|R_1| \leq M^T_7 (t-s)^2 \ \ \ \forall \ n \in \mathbb{N} \ \mbox{and } t,s \in[0,\ T],
\end{equation*}
where $$M^T_7 = M_3 T^{2p-2} + \sum_{d=0}^{p-2} M^d_6 T^{2d}.$$
This complete the proof of Step 1.
\vskip8pt
\noindent \textbf{Step 2: Proof of (\ref{eqn:|R_2|}).}
\vskip8pt
We first recall the term $R_2$ from (\ref{eqn:|R_1|+|R_2|}),
$$R_2 = 2^3 \big[ \frac{1}{\sqrt{n}} \E [\Tr(C_n(t))^p - \Tr(C_n(s))^p]\big]^4.$$
Now by using (\ref{eqn:C_n^p(t)- C_n^p(s)}) in the above equation of $R_2$, we get
\begin{align*}
R_2 & = 2^3 \big[ \frac{1}{\sqrt{n}} \E[\Tr (C'_n(t-s) )^p] + \frac{1}{\sqrt{n}} \sum_{d=1}^{p-1} \binom{p}{d}\E[\Tr (C'_n(t-s))^d (C_n(s))^{p-d}]  \big]^4,
\end{align*}
where $C'_n(t-s)= C_n(t)- C_n(s)$. Now using (\ref{eqn:E[C'_n]}), $\E[C'_n(t-s)]= \E[C_n(t-s)]$ and $\binom{p}{d} \leq 2^p$ for all $1 \leq d \leq p-1$, we get
\begin{align} \label{eqn:R_2}
R_2 & \leq 2^3\big[ \frac{1}{\sqrt{n}} \E [\Tr (C_n(t-s) )^p] +  \sum_{d=1}^{p-1} \frac{2^p}{\sqrt{n}} \E [\Tr (C_n(t-s))^d (C_n(s))^{p-d}]  \big]^4 \nonumber \\
& = 2^3 \big[ W_0 + \sum_{d=1}^{p-1} W_d \big]^4, \ \mbox{say}. 
\end{align}
We first calculate the term $W_0$ of (\ref{eqn:R_2}).
\begin{align}\label{eqn:W_01}
W_0 & = \frac{1}{\sqrt{n}} \E [\Tr (C_n(t-s) )^p] \nonumber \\
& = \frac{1}{\sqrt{n}} \E [\frac{n}{ n^\frac{p}{2}} \sum_{ A_p} b_{i_1}(t-s) b_{i_2} (t-s) \cdots b_{i_p} (t-s) ] \nonumber \\
& = \frac{1}{ n^{\frac{p-1}{2} } } \sum_{ A_p} \E[ b_{i_1}(t-s) b_{i_2} (t-s) \cdots b_{i_p} (t-s) ].
\end{align}
For a non-zero contribution from $\E[ b_{i_1}(t-s) b_{i_2} (t-s) \cdots b_{i_p} (t-s) ]$, no random variables can appear odd number of times as $b_i(t)$ is the standard Brownian motion. Suppose there are $m$ many distinct entries in $(i_1, i_2, \ldots , i_p)$ with $m<p$ and $q^{th}$ distinct entry appears exactly $2k_q$ times, for $1 \leq q \leq m$. Then
\begin{equation} \label{eqn:p=}
2k_1 + 2k_2 + \cdots + 2k_m = p.
\end{equation} 
Since For a non-zero contribution from $\E[ b_{i_1}(t-s) b_{i_2} (t-s) \cdots b_{i_p} (t-s) ]$, each random variables has to appear at least twice, therefore $2k_q  \geq 2$ for all $1 \leq q \leq m$. Hence
\begin{align} \label{eqn:p=>}
p= k_1  + 2k_2 + \cdots + 2k_m  \geq 2m.
\end{align}
Therefore from (\ref{eqn:W_01}), we get
\begin{align}
W_0 & = \frac{1}{ n^{\frac{p-1}{2} } } \sum_{ A_p} \E[ (b_{i_1}(t-s))^{2k_1} (b_{i_2} (t-s))^{2k_2} \cdots (b_{i_m} (t-s))^{2k_m} ]. \nonumber
\end{align}
As $2k$-th moment of $b_i(t)$ is $(t^k \frac{(2k)!}{2^k k!})$, from the above equation we get
\begin{align} \label{eqn:W_02}
W_0 & = (t-s)^{(k_1 + k_2 + \cdots + k_m)} g_0(p) O \biggl( \frac{ n^{m-1} }{ n^{\frac{p-1}{2} } } \biggr) \nonumber \\
& = (t-s)^{\frac{p}{2}} g_0(p) O(n^{ \frac{2m-p -1}{2} }) \nonumber \\
& = (t-s)^{\frac{p}{2}} g_0(p) h_0(n),
\end{align}
where $g_0(p)$ is some function of $p$ and $h_0(n) = O(n^{ \frac{2m-p -1}{2} })$. Note that $\lim_{n\tends \infty} h_0(n) = 0$ as $\frac{2m-p -1}{2} < 0$ from (\ref{eqn:p=>}). Therefore there exists a positive constant $M_8$, which depends only on $p$, such that

$$ | g_0(p) h_0(n)| \leq M_8  \ \ \forall \ n \in \mathbb{N}.$$
Hence
\begin{equation} \label{eqn:W_0}
W_0  \leq M_8 (t-s)^{\frac{p}{2}}  \ \ \forall \ n \in \mathbb{N}. 
\end{equation}
Now we calculate the term $W_d$ of (\ref{eqn:R_2}), for $1 \leq d \leq p-1.$ First recall
\begin{align*}
W_d & = \frac{2^p}{\sqrt{n}} \E [\Tr (C_n(t-s))^d (C_n(s))^{p-d}] \\
& = \frac{2^p}{\sqrt{n}} \E [\frac{n}{ n^\frac{p}{2}} \sum_{ A_p} b_{i_1}(t-s) b_{i_2} (t-s) \cdots b_{i_d} (t-s)  b_{i_{d+1}}(s) b_{i_{d+2}} (s) \cdots b_{i_p} (s) ] \nonumber \\
& = \frac{2^p}{ n^{\frac{p-1}{2} } } \sum_{ A_p} \E[  b_{i_1}(t-s) b_{i_2} (t-s) \cdots b_{i_d} (t-s)  b_{i_{d+1}}(s) b_{i_{d+2}} (s) \cdots b_{i_p} (s) ].
\end{align*}
Since $b_i(t-s)$ and $b_i(s)$ are independent for any $i \geq 0$. We have
\begin{align*}
W_d = \frac{2^p}{ n^{\frac{p-1}{2} } } \sum_{ A_p} \E[  b_{i_1}(t-s) b_{i_2} (t-s) \cdots b_{i_d} (t-s)] \E[b_{i_{d+1}}(s) b_{i_{d+2}} (s) \cdots b_{i_p}(s) ].
\end{align*}
Now by the similar calculation, as we have done for $W_0$, we get
\begin{align*}
W_d  = 2^p (t-s)^{\frac{d}{2}} s^{\frac{p-d}{2}} g_d(p) h_d(n),
\end{align*}
where $g_d(p)$ is some function of $p$ and $|h_d(n)| \leq |h_0(n)|$, $h_0(n)$ as in (\ref{eqn:W_02}). Therefore $h_d(n) \tends 0$ as $n \tends \infty $. Since $|s| \leq T$, therefore there exists a positive constant $M^d_9$, which depends only on $p$, $d$ and $T$, such that 
\begin{equation*}
|2^p s^{\frac{p-d}{2}} g_d(p) h_d(n)| \leq M^d_9 \ \ \forall \ n \in \mathbb{N}.
\end{equation*}
Hence
\begin{equation} \label{eqn:W_d}
W_d  \leq M^d_9 (t-s)^{\frac{d}{2}}  \ \ \forall \ n \in \mathbb{N}. 
\end{equation}
Now, using (\ref{eqn:W_0}) and (\ref{eqn:W_d}) in (\ref{eqn:R_2}), we get 
\begin{align}
R_2 & \leq 2^3 \big[ M_8 (t-s)^{\frac{p}{2}} + \sum_{d=1}^{p-1} M^d_9 (t-s)^{\frac{d}{2}} \big]^4 \nonumber \\
 & = 2^3 (t-s)^2 \big[ M_8 (t-s)^{\frac{p-1}{2}} + \sum_{d=0}^{p-2} M^d_9 (t-s)^{\frac{d}{2}} \big]^4  \ \ \ \forall \ n \in \mathbb{N}. \nonumber 
\end{align}
As $s,t \in [0, \ T]$, we get 
\begin{equation*}
|R_2| \leq M^T_{10} (t-s)^2 \ \ \ \forall \ n \in \mathbb{N} \ \mbox{and } t,s \in[0,\ T],
\end{equation*}
where 
$$M^T_{10} =  2^3 \big[M_8 T^{\frac{p-1}{2}} + \sum_{d=0}^{p-2} M^d_9 T^{\frac{d}{2}} \big]^4.$$

This complete the proof of Step 2.
\vskip5pt
Now from (\ref{eqn:|R_1|+|R_2|}), (\ref{eqn:|R_1|}) and (\ref{eqn:|R_2|}), we get 
\begin{align} \label{eqn:M_T}
\E[w_p(t) - w_p(s)]^4 & \leq M_T (t-s)^2 \ \ \ \forall \ n \in \mathbb{N} \ \mbox{and } t,s \in[0,\ T],
\end{align}
where $M_T = M^T_7 + M^T_{10}$, which depends only on $T$ and $p$. This completes the proof of Proposition \ref{pro:tight} for $\alpha = 4$ and $\beta = 1.$
\end{proof}
\begin{remark}
	Since the constant $M_T$ of (\ref{eqn:M_T}) depends on $p$, it may tends to infinity as $p \tends \infty$. Therefore, Theorem \ref{thm:process} may not be true for all $p \in \mathbb{N}$. But from the proof of Theorem \ref{thm:process}, we can conclude that as $n \tends \infty$
	\begin{equation*}
	\{ w_p(t) ; t \geq 0, 2 \leq p \leq N\} \stackrel{\mathcal D}{\rightarrow} \{N_p(t) ; t \geq 0 , 2 \leq p \leq N\}, 
	\end{equation*}  
	for some fixed $N \in \mathbb{N}$.
\end{remark}

\noindent{\bf Acknowledgement:} We would like to thank Prof. Arup Bose and Prof. K. Suresh Kumar for their comments.


\providecommand{\bysame}{\leavevmode\hbox to3em{\hrulefill}\thinspace}
\providecommand{\MR}{\relax\ifhmode\unskip\space\fi MR }
\providecommand{\MRhref}[2]{%
  \href{http://www.ams.org/mathscinet-getitem?mr=#1}{#2}
}
\providecommand{\href}[2]{#2}

\end{document}